\documentclass[11pt]{amsart}

\usepackage[text={420pt,650pt},centering]{geometry}

\usepackage{color}
\usepackage{esint,amssymb,amsmath}
\usepackage{graphicx}
\usepackage{MnSymbol}
\usepackage{mathtools}
\usepackage[colorlinks=true, pdfstartview=FitV, linkcolor=blue, citecolor=blue, urlcolor=blue,pagebackref=false]{hyperref}
\usepackage{microtype}

\usepackage{bm}
\usepackage{scalerel} 
\usepackage{dsfont}
\usepackage{stmaryrd}
\usepackage[font={footnotesize}]{caption}

\usepackage{mathrsfs}
\usepackage{stmaryrd}

\newtheorem{proposition}{Proposition}
\newtheorem{theorem}[proposition]{Theorem}
\newtheorem{lemma}[proposition]{Lemma}
\newtheorem{corollary}[proposition]{Corollary}

\theoremstyle{definition}
\newtheorem{remark}[proposition]{Remark}

\newcommand{\cref}[1]{Corollary~\ref{c.#1}}

\numberwithin{equation}{section}
\numberwithin{proposition}{section}

\renewcommand{\leq}{\leqslant}
\renewcommand{\geq}{\geqslant}
\newcommand{\Z}{\mathbb{Z}}

\newcommand{\R}{\mathbb{R}}
\newcommand{\E}{\mathbb{E}}
\renewcommand{\P}{\mathbb{P}}

\newcommand{\ep}{\varepsilon}

\renewcommand{\subset}{\subseteq}

\newcommand{\Id}{\mathsf{Id}}

\renewcommand{\subset}{\subseteq}

\DeclareMathOperator{\dist}{dist}

\DeclareMathOperator{\tr}{Tr}

\renewcommand{\tilde}{\widetilde}

\newcommand{\de}{\delta}

\newcommand{\dk}{\delta_k}

\renewcommand{\hat}{\widehat}

\newcommand{\C}{\mathbb{C}}

\newcommand{\A}{\mathscr{A}}
\renewcommand{\S}{\mathscr{S}}
\newcommand{\im}{\text{\rm Im}\hspace{0.1cm}}
\newcommand{\re}{\text{\rm Re}\hspace{0.1cm}}
\renewcommand{\dk}{\mathsf{d_K}}
\renewcommand{\i}{\text{\rm i}}
\renewcommand{\u}{\mathfrak{u}}
\renewcommand{\d}{\mathrm{d}}
\renewcommand{\v}{\mathfrak{v}}

\newcommand{\fc}{\mathsf{fc}}
\newcommand{\MP}{\mathsf{MP}}
\newcommand{\TW}{\mathsf{TW}}
\newcommand{\WS}{\mathsf{Wishart}}
\newcommand{\Xv}{X^{\mathrm{v}}}
\newcommand{\Xw}{X^{\mathrm{w}}}
\newcommand{\Hv}{H^{\mathrm{v}}}
\newcommand{\Hw}{H^{\mathrm{w}}}
\newcommand{\Ev}{\E^{\mathrm{v}}}
\newcommand{\Ew}{\E^{\mathrm{w}}}

\newcommand{\Av}{A^{\mathrm{v}}}

\newcommand{\Rv}{\hat{R}_{\mathrm{v}}}
\newcommand{\Rw}{\hat{R}_{\mathrm{w}}}
\newcommand{\Omv}{\Omega_{\mathrm{v}}}
\newcommand{\Omw}{\Omega_{\mathrm{w}}}

\newcommand{\la}{\langle}
\newcommand{\ra}{\rangle}

\begin{document}

\title[Quantitative Edge Universality for Sample Covariance Matrices]{Quantitative Universality for the Largest Eigenvalue of Sample Covariance Matrices}

\begin{abstract}
We prove the first explicit rate of convergence to the Tracy-Widom distribution for the fluctuation of the largest eigenvalue of sample covariance matrices that are not integrable. Our primary focus is matrices of type $ X^*X $ and the proof follows the Erd\"{o}s-Schlein-Yau dynamical method. We use a recent approach to the analysis of the Dyson Brownian motion from \cite{bourgade2018extreme} to obtain a quantitative error estimate for the local relaxation flow at the edge. Together with a quantitative version of the Green function comparison theorem, this gives the rate of convergence.

Combined with a result of Lee-Schnelli \cite{lee2016tracy}, some quantitative estimates also hold for more general separable sample covariance matrices $ X^* \Sigma X $ with general diagonal population $ \Sigma $.
\end{abstract}

\author[H. Wang]{Haoyu Wang}
\address[H. Wang]{Courant Institute of Mathematical Sciences, New York University, 251 Mercer St., New York, NY 10012}
\email{hw1973@nyu.edu}

\maketitle
\setcounter{tocdepth}{1}

\section{Introduction}
\subsection{Overview and main results}
Edge universality of sample covariance matrices has been a classical problem in random matrix theory. It is well known that the distribution of the largest eigenvalue (after appropriate rescaling) converges to the Tracy-Widom distribution. Early non-quantitative results were first proved in \cite{peche2009universality,pillai2014universality,Soshnikov}. Quantitative estimates, however, were only obtained for the Wishart ensemble (see \cite{el2006rate,Johansson,Johnstone,ma2012accuracy}), which is essentially an integrable model. In this paper, we prove the explicit rate of convergence $ N^{-2/9} $ to the Tracy-Widom distribution for all sample covariance matrices of type $ X^*X $ with general distributed entries, and an analogous result with deterministic rate $ N^{-1/57} $ for all separable sample covariance matrices $ X^* \Sigma X $ with diagonal population $ \Sigma $. For simplicity and motivations from statistics, we only consider the real case, but the whole proof and the results are also true for complex sample covariance matrices.

Let $ X=(x_{ij}) $ be an $ M \times N $ data matrix with independent real valued entries with mean 0 and variance $ M^{-1} $,
\begin{equation}\label{e.Assumption1}
x_{ij} = M^{-1/2}q_{ij},\ \ \ \E q_{ij}=0,\ \ \ \E q_{ij}^2 =1.
\end{equation}
Furthermore, we assume the entries $ q_{ij} $ have a sub-exponential decay, that is, there exists a constant $ \theta>0 $ such that for $ u>1 $,
\begin{equation}\label{e.Assumption2}
\P (|q_{ij}| > u) \leq \theta^{-1} \exp (-u^\theta).
\end{equation}
This sub-exponential decay assumption is mainly for convenience, other conditions such as the finiteness of a sufficiently high moment would be enough. (For a necessary and sufficient condition for the edge universality we refer to \cite{ding2018necessary}.)

The sample covariance matrix corresponding to data matrix $ X $ is defined by $ H := X^* X $. Throughout this paper, to avoid trivial eigenvalues, we will be working in the regime
$$ \xi=\xi(N) := N/M,\ \ \ \lim_{N \to \infty} \xi \in (0,1) \  \mbox{or}\ \xi \equiv 1. $$
We will mainly work with the rectangular case $ 0< \xi <1 $, but will also show how to adapt the arguments to the square case $ M \equiv N $. (The reason why we do not discuss the general case $ \lim \xi =1 $ is merely technical due to the lack of local laws at the hard edge. In particular, the rigidity estimate at the hard edge is only known for a fixed $ \xi \equiv 1 $ but not for $ \xi = \xi(N) \to 1 $.)

We order the eigenvalues of $ H $ as $ \lambda_1 \leq \cdots \leq \lambda_N $, and use $ \lambda_+ $ to denote the typical location of the largest eigenvalue (see \eqref{e.EndPoints} for the definition). For the main result of this paper, we consider the Kolmogorov distance
$$ \dk(X,Y) := \sup_{x} \left| \P(X \leq x) - \P(Y \leq x) \right|. $$

\begin{theorem}\label{t.Rate}
Let $ H_N $ be sample covariance matrices satisfying \eqref{e.Assumption1} and \eqref{e.Assumption2}. Let $ \TW $ be the Tracy-Widom distribution. For any $ \ep>0 $, for large enough $ N $ we have
$$ \dk(N^{2/3}(\lambda_N - \lambda_+),\TW) \leq N^{-\frac{2}{9}+\ep}. $$
\end{theorem}

The null case $ X^*X $ is our primary concern in this paper, but quantitative estimates are also valid for general diagonal population matrices $ X^* \Sigma X $ thanks to the comparison theorem for the Green function flow by Lee and Schnelli (see Section \ref{s.GeneralPopulation} for more details). Combining our quantitative edge universality for the null case (Theorem \ref{t.Rate}) with the Green function comparison by Lee-Schnelli (Proposition \ref{p.GreenFunctionFlow}), we derive the rate of convergence to Tracy-Widom distribution for separable sample covariance matrices with general diagonal population.

\begin{corollary}\label{t.RateGeneralPopulation}
Let $ Q := X^* \Sigma X $ be an $ N \times N $ separable sample covariance matrix, where $ X $ is an $ M \times N $ real random matrix satisfying \eqref{e.Assumption1} and \eqref{e.Assumption2}, and $ \Sigma $ is a real diagonal $ M \times M $ matrix satisfying \eqref{e.AssumptionPopulation}. Let $ \mu_N $ be the largest eigenvalue of $ Q $. For any $ \ep>0 $, for large enough $ N $ we have
\begin{equation}\label{e.RateGeneralPopulation}
\dk \left( \gamma_0 N^{2/3} (\mu_N - E_+),\TW \right) \leq N^{-\frac{1}{57} + \ep},
\end{equation}
where $ E_+ $ defined in \eqref{e.RightEndpoint} denotes the rightmost endpoint of the spectrum and $ \gamma_0 $ is a normalization constant defined in \eqref{e.ScalingConstant}.
\end{corollary}

The method of the paper follows the three-step strategy of the Erd\"{o}s-Schlein-Yau dynamical approach \cite{erdHos2011universality}: (i) a priori bounds on locations of eigenvalues; (ii) local relaxation of the eigenvalue dynamics; (iii) a density argument showing eigenvalues statistics have not changed after short time.

Specifically, in this paper, the three-step strategy is employed in the following way: (i) is the rigidity for singular values, which can be rephrased from classical results on eigenvalues (see \cite{bloemendal2014isotropic,pillai2014universality}). For the particular square case $ M \equiv N $, we use a different rigidity estimate at the hard edge from \cite{alt2017local},which results in a slightly different proof; (ii) is the recent approach to the analysis of Dyson Brownian motion from \cite{bourgade2018extreme}, which introduced an observable defined via interpolation with integrable models (see \cite{bourgade2016fixed,landon2019fixed}). It describes the singular values evolution through a stochastic advection equation; (iii) is a quantitative version of the Green function comparison theorem, which is a slight extension of the classical result from \cite{GreenFunctionComparison}.

\begin{remark}
As discussed in \cite[Section 3]{pillai2014universality}, though we discuss the problems in the context of covariance matrices, the proof should also work for more generalized problems such as the quantitative edge universality of correlation matrices (see \cite{pillai2012edge}).
\end{remark}

This paper is organized as follows. The main part of the paper (Section \ref{s.prelim}-\ref{s.Comparison}) is devoted to the null case $ X^*X $. In Section \ref{s.prelim} we rephrase the classical results on the eigenvalues of sample covariance matrices to the version for singular values, including Dyson Brownian motion, local laws and rigidity estimates. In Section \ref{s.LocalRelaxationFlow}, we define an observable that describes the evolution of singular values, and then prove the error estimate for the local relaxation flow at the edge by studying the dynamics of the observable. In Section \ref{s.Comparison} we prove the quantitative Green function comparison theorem and use it to derive the rate of convergence to the Tracy-Widom distribution. Finally, in Section \ref{s.GeneralPopulation} we generalize our result to separable sample covariance matrices with diagonal population by studying the interpolation between general covariance matrices with the null case.

\subsection{Notations}
Throughout this paper, we use the notation $ A \lesssim B $ if there exists a constant $ C $ which is independent of $ N $ such that $ A \leq CB $ holds. We also denote $ A \sim B $ if both $ A \lesssim B $ and $ B \lesssim A $ hold. If $ A $ and $ B $ are complex valued, $ A \sim B $ means $ \re A \sim \re B $ and $ \im A \sim \im B $. We also denote $ C $ a generic constant which does not depend on $ N $ but may vary form line to line. We use $ \llbracket A,B \rrbracket := [A,B] \cap \Z $ to denote the set of integers between $ A $ and $ B $.

We also denote
$$ \varphi = e^{C_0 (\log \log N)^2} $$
a subpolynomial error parameter, for some fixed $ C_0>0 $. This constant $ C_0 $ is chosen large enough so that the eigenvalues (and singular values) rigidity and the strong local Marchenko-Pastur law hold (see section \ref{s.LocalLaw}).

\subsection*{Acknowledgement}
The author would like to thank Prof. Paul Bourgade for suggesting this problem, helpful discussions, and useful comments on the early draft of the paper.

\section{Preliminaries}\label{s.prelim}
\subsection{Dyson Brownian motion for covariance matrices}
Let $ B $ be an $ M \times N $ real matrix Brownian motion: $ B_{ij} $ are independent standard Brownian motions. We define the $ M \times N $ matrix $ M_t $ by
$$ M_t = M_0 + \dfrac{1}{\sqrt{N}}B_t. $$
The eigenvalues dynamics for the real Wishart process $ X_t := M_t^* M_t $ was first proved in \cite{bru1989diffusions}. Under our normalization convention, the equation is in the following form given in \cite[Appendix C]{bourgade2017eigenvector}
\begin{equation}\label{e.EigenvalueDBM}
\d \lambda_k = 2\sqrt{\lambda_k}\dfrac{\d B_{kk}}{\sqrt{N}} + \left( \dfrac{M}{N} + \dfrac{1}{N} \sum_{l \neq k}\dfrac{\lambda_k + \lambda_l}{\lambda_k - \lambda_l} \right)\d t.
\end{equation}
Due to technical issues, it is difficult to use the coupling method from \cite{bourgade2016fixed} to analyze \eqref{e.EigenvalueDBM} in a direct way. This motivates us to consider the singular values instead.

Let $ s_k := \sqrt{\lambda_k} $ denote the singular values of $ X $. The Dyson Brownian motion for singular values dynamics of such sample covariance matrices is the following Ornstein-Uhlenbeck process \cite[equation (5.8)]{erdHos2012local}.
$$ \d s_k = \dfrac{\d B_k}{\sqrt{N}} + \left[ -\dfrac{1}{2\xi}s_k + \dfrac{1}{2}\left( \dfrac{1}{\xi} -1 \right) \dfrac{1}{s_k} + \dfrac{1}{2N} \sum_{l \neq k} \left( \dfrac{1}{s_k - s_l} + \dfrac{1}{s_k + s_l} \right) \right]\d t, \ \ \ 1 \leq k \leq N. $$
An important idea in this paper is the following symmetrization trick (see \cite[equation (3.9)]{che2019universality}):
$$ s_{-i}(t)=-s_i(t),\ \ \ B_{-i}(t)=-B_i(t),\ \ \forall t \geq 0,\ 1 \leq i \leq N. $$
From now we label the indices from $ -1 $ to $ -N $ and $ 1 $ to $ N $, so that the zero index is omitted. Unless otherwise stated, this will be the convention and we will not emphasize it explicitly. After symmetrization, the dynamics turns to the following form
\begin{equation}\label{e.SymmetrizedDynamics}
\d s_k = \dfrac{\d B_k}{\sqrt{N}} + \left[ -\dfrac{1}{2\xi}s_k + \dfrac{1}{2}\left( \dfrac{1}{\xi} -1 \right) \dfrac{1}{s_k} + \dfrac{1}{2N} \sum_{l \neq \pm k} \dfrac{1}{s_k - s_l} \right]\d t, \ \ \ -N \leq k \leq N, k \neq 0.
\end{equation}

\subsection{Local law and rigidity for singular values}\label{s.LocalLaw}
The local law and rigidity estimates are classical results for the eigenvalues of sample covariance matrices. In this section, for later use, we rephrase these results into the corresponding version in terms of singular values.

It is well known that the empirical measure of the eigenvalues converges to the Marchenko-Pastur distribution
$$ \rho_{\MP}(x) = \dfrac{1}{2\pi \xi}\sqrt{\dfrac{[(x-\lambda_-)(\lambda_+ -x)]_+}{x^2}}, $$
where
\begin{equation}\label{e.EndPoints}
\lambda_{\pm} = (1 \pm \sqrt{\xi})^2.
\end{equation}
Define the typical locations of the singular values:
$$ \gamma_k :=  \inf \left\{ E>0 : \int_{-\infty}^{E^2} \rho_{\MP}(x)\d x \geq \dfrac{k}{N} \right\},\ \ \ 1 \leq k \leq N. $$
Following the symmetrization trick, we also define $ \gamma_{-k} = -\gamma_k $. By a change of variable, it is easy to check that
\begin{equation}\label{e.typical}
\int_{-\infty}^{\gamma_k} \rho(x)\d x=\dfrac{N+k}{2N}, \ \ \ \int_{-\infty}^{\gamma_{-k}} \rho(x)\d x = \dfrac{N-k}{2N},
\end{equation}
where $ \rho(x) $ is the counterpart of Marchenko-Pastur law for singular values, defined by
\begin{equation}\label{e.measure}
\rho(x)=\dfrac{1}{2\pi \xi}\sqrt{\dfrac{[(x^2-\lambda_-)(\lambda_+ - x^2)]_+}{x^2}},\ \ \ \sqrt{\lambda_-} \leq |x| \leq \sqrt{\lambda_+}.
\end{equation}

Denote $ s_1 \leq \cdots \leq s_N $ the singular values of the data matrix $ X $, and extend the singular values following the symmetrization trick by $ s_{-k}=-s_k $. For $ z=E+\i\eta \in \mathbb{C} $ with $ \eta>0 $, let $ m_N(z) $ and $ S_N(z) $ denote the Stieltjes transform of the empirical measure of the (symmetrized) singular values and eigenvalues, respectively:
$$ m_N(z):= \dfrac{1}{2N}\sum_{-N \leq k \leq N} \dfrac{1}{s_k -z},\ \ \ S_N(z) := \dfrac{1}{N}\sum_{k=1}^N \dfrac{1}{\lambda_k -z}. $$
As mentioned previously, in the summation from $ -N $ to $ N $ the $ 0 $ index is always excluded. Note that due to the symmetrization, this is equivalent to
\begin{equation}\label{e.DiscreteStieltjes}
 m_N(z) = \dfrac{1}{2N}\sum_{k=1}^N \left( \dfrac{1}{s_k -z} + \dfrac{1}{-s_k - z} \right) = \dfrac{1}{N} \sum_{k=1}^N \dfrac{z}{s_k^2 - z^2} = z S_N(z^2).
\end{equation}
On the other hand, use $ m_{\MP}(z) $ to denote the Stieltjes transform of the Marchenko-Pastur law
$$ m_{\MP}(z) := \int_{\R} \dfrac{\rho_{\MP}(x)}{x-z}\d x=\dfrac{1-\xi-z+\sqrt{(z-\lambda_-)(z-\lambda_+)}}{2\xi z}, $$
where $ \sqrt{\quad} $ denotes the square root on the complex plane whose branch cut is the negative real line. With this choice we always have $ \im m_{\MP}(z)>0 $ when $ \im z>0 $. For the singular values, recall the limit distribution $ \rho(x) $ for the empirical measure and use $ m(z) $ to denote its corresponding Stieltjes transform
$$ m(z) := \int_{\R} \dfrac{\rho(x)}{x-z}\d x = \int_{\R} \dfrac{1}{x-z}\dfrac{1}{2\pi \xi}\sqrt{\dfrac{[(x^2-\lambda_-)(\lambda_+ - x^2)]_+}{x^2}}\d x. $$
We have the following relation between $ m(z) $ and $ m_{\MP}(z) $
\begin{multline}\label{e.ContinuousStieltjes}
m(z) = \int_{\sqrt{\lambda_-}}^{\sqrt{\lambda_+}} \left( \dfrac{1}{x-z} - \dfrac{1}{x+z} \right) \dfrac{1}{2\pi \xi}\sqrt{\dfrac{[(x^2-\lambda_-)(\lambda_+ - x^2)]_+}{x^2}}\d x\\
= \int_{\R}\dfrac{z}{x-z^2}\rho_{\MP}(x)\d x = zm_{\MP}(z^2).
\end{multline}
It is well known that we have the strong local Marchenko-Pastur law \cite{bloemendal2014isotropic,pillai2014universality} for the estimate of $ S_N(z) $, i.e. for any $ D>0 $, there exists $ N_0(D)>0 $ such that for every $ N \geq N_0 $ we have
$$ \P \left( \left| S_N(z) - m_{\MP}(z) \right| \leq\dfrac{\varphi}{N\eta} \right) > 1-N^{-D}. $$
By the relations \eqref{e.DiscreteStieltjes} and \eqref{e.ContinuousStieltjes}, we know that
$$ |m_N(z) - m(z)| = \left| z \left( S_N(z^2) - m_{\MP}(z^2) \right) \right| \leq |z| \left| S_N(z^2) - m_{\MP}(z^2) \right|. $$
Combining with the strong local Marchenko-Pastur law, this gives us
\begin{equation}\label{e.LocalLaw}
\P \left( \left| m_N(z) - m(z) \right| \lesssim \dfrac{\varphi}{N\eta} \right) > 1-N^{-D}.
\end{equation}

\begin{figure}[ht]
\includegraphics[height=4cm]{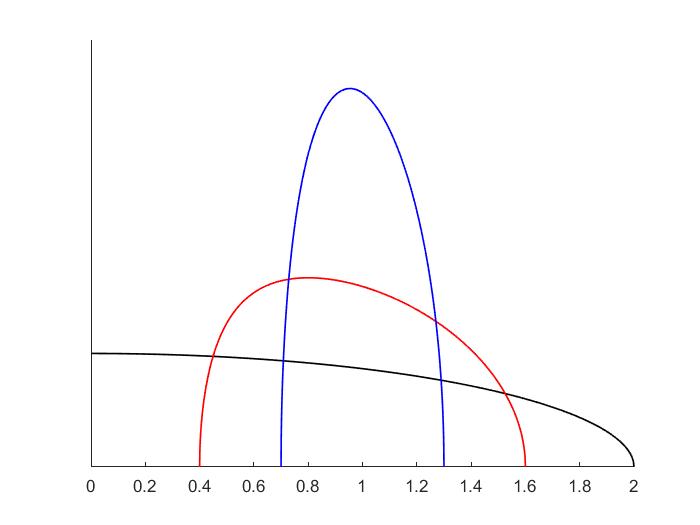}
\caption{Graphs of the density $ \rho(x) $ for $ x\geq0 $ (i.e. the density for the actual singular values) with $ \xi=1,(3/5)^2,(3/10)^2 $ respectively. The curve for $ \xi=1 $ does not have a "square root"-type shape at the edge $ \sqrt{\lambda_-} =0 $ in this case due to the singularity of the Marchenko-Pastur distribution.}
\end{figure}

For the rigidity estimates, a key observation is that the critical case $ \xi=1 $ is significantly different from other cases. This is because the Marchenko-Pastur law $ \rho_{\MP} $ has a singularity at the point $ x=0 $ in this situation. When $ \xi < 1 $, the rigidity of singular values can be easily obtained from the analogous estimates for eigenvalues (see \cite{pillai2014universality}). Let $ \hat{k}:=\min(k,N+1-k) $, for any $ D>0 $ there exists $ N_0(D) $ such that the following holds for any $ N \geq N_0 $,
\begin{equation}\label{e.RigidityNotOne}
\P \left(|s_k - \gamma_k| \leq \varphi^{\frac{1}{2}} N^{-\frac{2}{3}}(\hat{k})^{-\frac{1}{3}} \  \mbox{for all}\ k \in \llbracket 1,N \rrbracket \right) > 1-N^{-D}.
\end{equation}

For the critical case $ \xi=1 $, now the Marchenko-Pastur distribution is supported on $ [0,4] $ and is given by $ \rho_{\MP}(x)= \frac{1}{2\pi}\sqrt{(4-x)/x} $. A key observation is that the scales of eigenvalue spacings are different at the two edges. Due to this phenomenon, we use the following two different results, depending on the location in the spectrum.

On the one hand, the Marchenko-Pastur distribution still behaves like a square root near the soft edge $ x=4 $, which implies that the result is the same as the rectangular case. The rigidity estimate near the soft edge can be easily adapted from the result for eigenvalues in \cite[Theorem 2.10]{bloemendal2014isotropic}, i.e. for some (small) $ \omega>0 $ and any $ \ep>0 $ we have
\begin{equation}\label{e.RigidityOneSoft}
\P \left( |s_k - \gamma_k| \leq N^\ep (N-k+1)^{-\frac{1}{3}}N^{-\frac{2}{3}} \ \mbox{for all} \ k \in \llbracket (1-\omega)N,N \rrbracket \right) > 1-N^{-D}.
\end{equation}

On the other hand, as explained in \cite{cacciapuoti2013local}, at the hard edge $ x=0 $ the typical distance between eigenvalues and the edge is of order $ N^{-2} $, which is much smaller than the typical distance between neighbouring eigenvalues in the bulk (or at the soft edge). Note that in this situation, the measure for the symmetrized singular values coincides with the standard semicircle law, that is $ \rho(x) = \frac{1}{2\pi}\sqrt{(4-x^2)_+} $. By the relation \eqref{e.typical}, this means that the typical $ k $-th singular value $ \gamma_k $ of an $ N \times N $ data matrix shares the same position with the typical $ (N+k) $-th eigenvalue of a $ 2N \times 2N $ generalized Wigner matrix. The link between these two models can be illustrated by the symmetrization trick: Define the $ 2N \times 2N $ matrix
\begin{equation*}
\tilde{H} = \left(
\begin{matrix}
0 & X^*\\
X & 0
\end{matrix}
\right),
\end{equation*}
then we know that the eigenvalues of $ \tilde{H} $ are precisely the symmetrized singular values of $ X^* $. Note that we have $ \tilde{H}=\tilde{H}^* $, $ \E \tilde{H}_{ij}=0 $ and $ \sum_{i=1}^{2N} \E \tilde{H}_{ij}^2 = 1 $ for every $ j \in \llbracket 1,2N \rrbracket $. This shows that $ \tilde{H} $ is indeed a Wigner-type matrix except the lack of nondegeneracy condition caused by the zero blocks. By considering the matrix of this type, the rigidity at the hard edge can be proved directly from \cite[Theorem 2.7]{alt2017local}
\begin{equation}\label{e.RigidityOneHard}
\P \left( |s_k - \gamma_k| \leq N^{-1+\ep} \ \mbox{for all}\ k \in \llbracket 1,(1-\omega)N \rrbracket \right) > 1-N^{-D}.
\end{equation}

\section{Stochastic Advection Equation for Singular Values Dynamics}\label{s.LocalRelaxationFlow}
\subsection{Stochastic advection equation}
We follow the comparison method via coupling in \cite{bourgade2016fixed}. As in \cite{landon2019fixed}, consider the interpolation between a general sample covariance matrix and the Wishart ensemble for the initial data: for any $ \nu \in [0,1] $, let
$$ x_k^{(\nu)}(0)=\nu s_k(0)+(1-\nu)r_k(0),\ \ \ -N \leq k \leq N,\ k \neq 0. $$
where $ s_k(t) $ and $ r_k(t) $ satisfy the singular values dynamics \eqref{e.SymmetrizedDynamics}, with respective initial conditions a general sample covariance matrix and the Wishart ensemble. Define the corresponding dynamics of $ x_k^{(\nu)} $ to be
\begin{equation}\label{e.DBM}
\d x_k^{(\nu)} = \dfrac{\d B_k}{\sqrt{N}} + \left[ -\dfrac{1}{2\xi}x_k^{(\nu)} + \dfrac{1}{2}\left( \dfrac{1}{\xi} -1 \right) \dfrac{1}{x_k^{(\nu)}} + \dfrac{1}{2N} \sum_{l \neq \pm k} \dfrac{1}{x_k^{(\nu)} - x_l^{(\nu)}} \right]\d t .
\end{equation}
For this Dyson Brownian motion we consider the quantity
$$ \u_k^{(\nu)}(t):=e^{\frac{t}{2\xi}}\dfrac{\d}{\d\nu}x_k^{(\nu)}(t). $$
From now we set $ \nu \in (0,1) $ and omit it from the notation for simplicity. A significant property is that $ \u_k $ satisfies a non-local parabolic differential equation
\begin{equation}\label{e.parabolic}
\dfrac{\d}{\d t}\u_{k}=\dfrac{1}{2}\left( 1- \dfrac{1}{\xi} \right)\dfrac{\u_k}{x_k^2} + \dfrac{1}{2N} \sum_{l \neq \pm k}\dfrac{\u_l - \u_k}{\left(x_l - x_k \right)^2}.
\end{equation}
Let $ \v_k=\v_k^{(\nu)} $ solve the same equation as $ \u_k $ in \eqref{e.parabolic} but with initial condition $ \v_k(0)=|\u_k(0)|=|s_k(0)-r_k(0)| $. An important result is that this equation yields a maximum principle for $ \v_k $, which will be useful in later analysis for the estimate of its growth.

\begin{lemma}\label{l.MaxPrin}
For all $ t \geq 0 $ and $ -N \leq k \leq N $, we have
$$ \v_k(t) \geq 0, \ \ \ |\v_k(t)| \leq \max_k |\v_k(0)|,\ \ \ |\u_k(t)| \leq \v_k(t). $$
\end{lemma}
\begin{proof}
Note that for the coefficients in the summation part of equation \eqref{e.parabolic} we have $ \frac{1}{(x_l - x_k)^2} >0 $. Therefore for $ f(t) := \min_k \v_k(t) $, we have
$$ f'(t) \geq \dfrac{1}{2} \left( 1-\dfrac{1}{\xi} \right) \dfrac{1}{x_1^2} f(t). $$
Combined with the fact $ \v_k(0) \geq 0 $, this gives us the first claim. For the second claim, since $ \v_k $'s are nonnegative, we know that $ \frac{\d}{\d t}\max_{k} \v_k  \leq 0 $ and this yields the desired result. The third claim follows from linearity. We know that both $ \v + \u $ and $ \v - \u $ satisfy the equation \eqref{e.parabolic}, and we also have $ (\v + \u)(0) \geq 0 $ and $ (\v - \u)(0) \geq 0 $. Similarly to the first claim, this gives us $ \v_k(t) + \u_k(t) \geq 0 $ and $ \v_k(t) - \u_k(t)  \geq 0 $, which completes the proof.
\end{proof}

\begin{remark}
Due to our choice for the initial value of $ \v $, we will have that $ \{\v_k\}_{-N \leq k \leq N} $ are symmetric with respect to the label $ k $. To see this, it is easy to check that $ \tilde{\v}_k := \v_{-k} $ satisfy the same equation \eqref{e.parabolic}. Note that the equation is linear and $ \tilde{\v}_k(0)=\v_k(0) $ for all $ k $. Lemma \ref{l.MaxPrin} then gives us $ \v_k(t) = \tilde{\v}_k(t) = \v_{-k}(t) $ for all $ k $ and $ t \geq 0 $.
\end{remark}

We now consider the observable
$$ f_t(z)=e^{-\frac{t}{2\xi}}\sum_{-N \leq k \leq N}\dfrac{\v_k(t)}{x_k(t)-z}. $$
A key observation is that the quadratic singularities in \eqref{e.parabolic} will disappear when combined with the Dyson Brownian motion, so that the evolution of $ f_t $ has no shocks similarly to a result from \cite{bourgade2018extreme}. Denoting
$$ s_t(z)=\dfrac{1}{2N}\sum_{-N \leq k \leq N}\dfrac{1}{x_k(t)-z} $$
the Stieltjes transform of the (symmetrized) empirical spectral measure, then the observable $ f_t $ satisfies the following dynamics.

\begin{lemma}\label{l.dynamics}
For any $  \im z \neq 0 $, we have
\begin{equation}\label{e.dynamics}
\begin{aligned}
\d f_t &=\left(s_t(z)+\dfrac{z}{2\xi}\right)(\partial_z f_t)\d t+\dfrac{1}{4N}(\partial_{zz}f_t)\d t+\left[\dfrac{e^{-\frac{t}{2\xi}}}{2N}\sum_{-N \leq k \leq N}\dfrac{\v_k}{(x_k-z)^2(x_k+z)}\right]\d t\\
& \quad + \left[\left( 1- \dfrac{1}{\xi} \right) e^{-\frac{t}{2\xi}} \left( \sum_{-N \leq k \leq N} \dfrac{3z \v_k}{2 x_k^2 (x_k -z)(x_k+z)} + \sum_{-N \leq k \leq N} \dfrac{z^3 \v_k}{x_k^2 (x_k-z)^2 (x_k+z)^2} \right)\right] \d t\\
& \quad -\dfrac{e^{-\frac{t}{2\xi}}}{\sqrt{N}}\sum_{-N \leq k \leq N}\dfrac{\v_k}{(x_k-z)^2}\d B_k.
\end{aligned}
\end{equation}
\end{lemma}

\begin{proof}
This can be proved by direct computation via the It\^{o}'s formula. First, we have
$$ \d f=-\dfrac{f}{2\xi}\d t+e^{-t/2}\sum_{-N \leq k \leq N}\dfrac{\d\v_k}{x_k-z}+e^{-t/2}\sum_{-N \leq k \leq N}\v_k \d\dfrac{1}{x_k-z} =: A_1 + A_2 +A_3. $$
By using It\^{o}'s formula again we have $ \d (x_k-z)^{-1}=-(x_k-z)^{-2}\d x_k+\frac{1}{N}(x_k-z)^{-3}\d t $. Thus, we can now decompose the term $ A_3 $ as $ I_1+[I_2+I_3+I_4+I_5]\d t $, where
\begin{align*}
I_1 &=
-\dfrac{e^{-\frac{t}{2\xi}}}{\sqrt{N}}\sum_{-N \leq k \leq N}\dfrac{\v_k}{(x_k-z)^2}\d B_k, \\
I_2 &=
\dfrac{e^{-\frac{t}{2\xi}}}{2\xi}\sum_{-N \leq k \leq N}\dfrac{\v_k x_k}{(x_k-z)^2}, \\
I_3 &=
\dfrac{1}{2} \left( 1-\dfrac{1}{\xi} \right) e^{-\frac{t}{2\xi}} \sum_{-N \leq k \leq N} \dfrac{\v_k}{x_k(x_k-z)^2}, \\
I_4 &= e^{-\frac{t}{2\xi}}\sum_{-N \leq k \leq N}\v_k\left(-\dfrac{1}{(x_k-z)^2}\right)\dfrac{1}{2N}\sum_{l \neq \pm k}\dfrac{1}{x_k-x_l}=\dfrac{e^{-\frac{t}{2\xi}}}{2N}\sum_{l \neq \pm k}\dfrac{\v_k}{(x_k-z)^2(x_l-x_k)}, \\
I_5 &= \dfrac{e^{-\frac{t}{2\xi}}}{N}\sum_{-N \leq k \leq N}\dfrac{\v_k}{(x_k-z)^3}.
\end{align*}
Note that
$$ \partial_z f=e^{-\frac{t}{2\xi}}\sum_{-N \leq k \leq N}\dfrac{\v_k}{(x_k-z)^2},\ \ \ \partial_{zz}f=2e^{-\frac{t}{2\xi}}\sum_{-N \leq k \leq N}\dfrac{\v_k}{(x_k-z)^3}. $$
For the term $ A_2 $, by the equation \eqref{e.parabolic}, we have
$$ A_2 = e^{-\frac{t}{2\xi}}\sum_{-N \leq k \leq N}\dfrac{1}{x_k-z} \left[ \dfrac{1}{2} \left( 1-\dfrac{1}{\xi} \right) \dfrac{\v_k}{x_k^2}  + \dfrac{1}{2N}\sum_{l \neq \pm k}\dfrac{\v_l-\v_k}{(x_k-x_l)^2} \right]=: B_1+B_2. $$
Note that
\begin{multline*}
B_2 = \dfrac{e^{-\frac{t}{2\xi}}}{4N}\sum_{l \neq \pm k}\dfrac{\v_l-\v_k}{(x_l-x_k)^2}\left(\dfrac{1}{x_k-z}-\dfrac{1}{x_l-z}\right)=\dfrac{e^{-\frac{t}{2\xi}}}{4N}\sum_{l \neq \pm k}\dfrac{\v_l-\v_k}{x_l-x_k}\dfrac{1}{x_k-z}\dfrac{1}{x_l-z}\\
=-\dfrac{e^{-\frac{t}{2\xi}}}{2N}\sum_{l \neq \pm k}\dfrac{\v_k}{x_l-x_k}\dfrac{1}{x_k-z}\dfrac{1}{x_l-z}.
\end{multline*}
Combining with $ I_4 $, we obtain
$$ B_2 + I_4 =\dfrac{e^{-\frac{t}{2\xi}}}{2N}\sum_{l \neq \pm k}\dfrac{\v_k}{x_l-x_k}\dfrac{1}{x_k-z}\left(\dfrac{1}{x_k-z}-\dfrac{1}{x_l-z}\right)=\dfrac{e^{-\frac{t}{2\xi}}}{2N}\sum_{l \neq \pm k}\dfrac{\v_k}{(x_k-z)^2}\dfrac{1}{x_l-z}. $$
Moreover, we have that
\begin{equation*}
B_2+I_4+I_5 = s(z)\partial_z f+\dfrac{e^{-\frac{t}{2\xi}}}{2N}\sum_{-N \leq k \leq N}\dfrac{\v_k}{(x_k-z)^3}+\dfrac{e^{-\frac{t}{2\xi}}}{2N}\sum_{-N \leq k \leq N}\dfrac{\v_k}{(x_k-z)^2(x_k+z)}.
\end{equation*}
Then it suffices to calculate $ B_1 + I_3 $, and note that
\begin{align*}
B_1 + I_3 &= \dfrac{1}{2} \left( 1- \dfrac{1}{\xi} \right) e^{-\frac{t}{2\xi}} \sum_{k=1}^N \left( \dfrac{\v_k}{x_k -z}\dfrac{1}{x_k^2} - \dfrac{\v_k}{x_k +z}\dfrac{1}{x_k^2} + \dfrac{\v_k}{(x_k - z)^2}\dfrac{1}{x_k} - \dfrac{\v_k}{(x_k +z)^2}\dfrac{1}{x_k} \right)\\
&= \left( 1- \dfrac{1}{\xi} \right) e^{-\frac{t}{2\xi}} \left( \sum_{-N \leq k \leq N} \dfrac{3z \v_k}{2 x_k^2 (x_k -z)(x_k+z)} + \sum_{-N \leq k \leq N} \dfrac{z^3 \v_k}{x_k^2 (x_k-z)^2 (x_k+z)^2} \right).
\end{align*}
Note that now all the singularities are removed.

The desired result then follows by combining the previous results and the term $ I_1 $.
\end{proof}

Recall that in Section \ref{s.LocalLaw} we have shown that the Stieltjes transform of the empirical measure for singular values satisfies the local law \eqref{e.LocalLaw}, so that the leading term of the stochastic differential equation \eqref{e.dynamics} satisfied by $ f_t $ is close to
$$ \dfrac{z}{2\xi} + z m_{\MP}(z^2) = \dfrac{z}{2\xi} + \dfrac{1-\xi-z^2 + \sqrt{(z^2 - \lambda_-)(z^2 - \lambda_+)}}{2\xi z} = \dfrac{(1-\xi) + \sqrt{(z^2 - \lambda_-)(z^2 - \lambda_+)}}{2\xi z}. $$
Thus, the dynamics of $ f_t $ can be approximated by the following advection equation
\begin{equation}\label{e.AdvectionPDE}
\partial_t r = \dfrac{(1-\xi) + \sqrt{(z^2 - \lambda_-)(z^2 - \lambda_+)}}{2\xi z} \partial_z r.
\end{equation}

\subsection{Geometric properties of the characteristics}\label{s.geometry}
In order to estimate the evolution of the observable, we analyze its dynamics \eqref{e.dynamics} by studying the characteristics of the approximate advection PDE \eqref{e.AdvectionPDE}, similarly to \cite{huang2018rigidity,bourgade2018extreme}. To do this, we first need some bounds on the shape of the characteristics $ (z_t)_{t \geq 0} $, and some estimates for the initial value. As mentioned in Section \ref{s.LocalLaw}, discussions for the case $ \xi=1 $ are expected to be different due to the singularity of the Marchenko-Pastur distribution (which results in a distinct shape of the density $ \rho(x) $ for singular values). Therefore, in this and the subsequent section, we first discuss the case $ \xi \neq 1 $ and will show how to adapt the proof to the case for square data matrices in Section \ref{s.CaseOne}.

Denote
$$ \kappa(z) := \min\left\{\left|z-\sqrt{\lambda_-}\right|,\left|z-\sqrt{\lambda_+}\right|\right\}, $$
and
$$ a(z) := \dist \left(z,\left[\sqrt{\lambda_-},\sqrt{\lambda_+}\right]\right),\ \ \ b(z):= \dist \left( z, \left[\sqrt{\lambda_-},\sqrt{\lambda_+}\right]^c \right) $$
We consider the curve
$$ \S := \left\{ z=E+\i y: \sqrt{\lambda_-} + \varphi^2 N^{-2/3} <E< \sqrt{\lambda_+} - \varphi^2 N^{-2/3},\ y=\varphi^2/\left( N \kappa(E)^{1/2} \right) \right\}, $$
and the domain $ \mathscr{R} := \cup_{0<t<1}\{z_t:z \in \S\} $.

\begin{lemma}\label{l.characteristics}
Uniformly in $ 0<t<1 $ and $ z=z_0 $ satisfying $ \eta := \im z>0 $ and $ |z-\sqrt{\lambda_+}|<\sqrt{\xi}/10 $, we have
$$ \re (z_t -z_0) \sim t \dfrac{a(z)}{\kappa(z)^{1/2}}+t^2,\ \ \ \ \im(z_t-z_0) \sim t\dfrac{b(z)}{\kappa(z)^{1/2}}. $$
In particular, if in addition we have $ z \in \S $, then
$$ (z_t - z_0) \sim \left( t \dfrac{\varphi^2}{N \kappa(E)} +t^2 \right) + \i \kappa(E)^\frac{1}{2} t.  $$
Moreover, for any $ \kappa >0 $, uniformly in $ 0<t<1 $ and $ z=E+\i \eta \in [\sqrt{\lambda_-} + \kappa,\sqrt{\lambda_+} -\kappa] \times [0,\kappa^{-1}] $, we have $ \im (z_t-z_0) \sim t $.
\end{lemma}
\begin{figure}[ht]
\includegraphics[height=4cm]{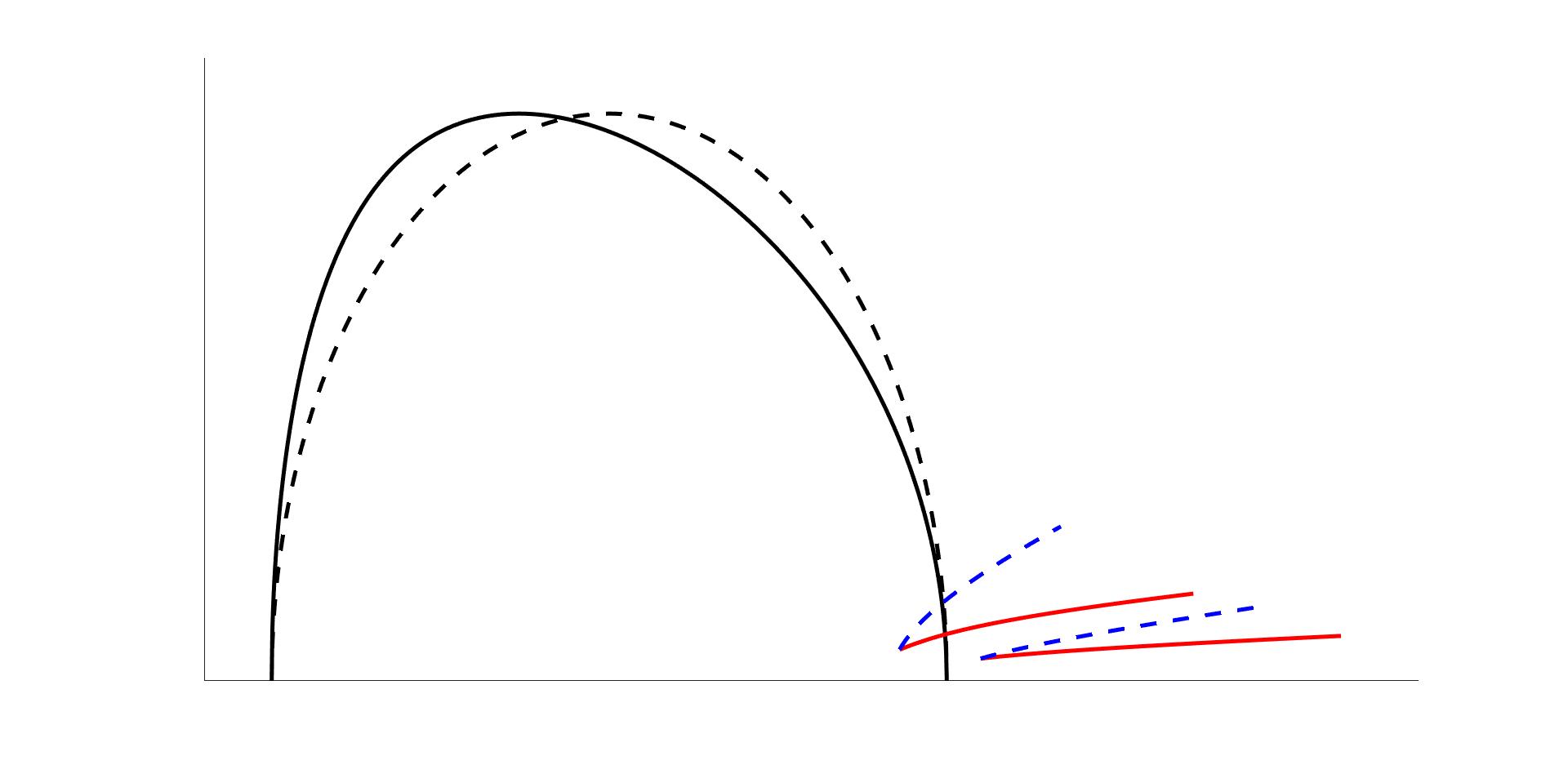}
\caption{Shape of the characteristics near the edge: the solid lines represent the density $ \rho(x) $ and the corresponding characteristics; the dashed lines represent the semicircle distribution and the related characteristics.}
\end{figure}
\begin{proof}
It is too complicated to work with the ODE satisfied by $ z_t $ in a direct way. The main idea is to compare this characteristics with the corresponding curve for a semicircle distribution (which has a explicit and simple formula). Define the two functions
$$ g(z):=\dfrac{(1-\xi) + \sqrt{(z^2 - \lambda_-)(z^2 - \lambda_+)}}{2\xi z} , \ \ \ g_{sc}(z):=\dfrac{\sqrt{(z-1)^2 - \xi}}{\xi}. $$
For some $ |z_0-\sqrt{\lambda_+}| < \sqrt{\xi}/10,\ \im z_0>0 $, let $ z(t) := z_t $ and $ z_{sc}(t) $ solve the following two initial value problems
\begin{equation*}
\left\{
\begin{aligned}
& \dfrac{\d z}{\d t} = g(z)\\
& z(0)=z_0
\end{aligned}
\right.
,\ \ \ \ \ \ \
\left\{
\begin{aligned}
& \dfrac{\d z_{sc}}{\d t} = g_{sc}(z_{sc})\\
& z_{sc}(0)=z_0
\end{aligned}
\right.
.
\end{equation*}
Note that in $ \Omega=\{z_{sc}(t):0<t<1,|z_0-\sqrt{\lambda_+}| < \sqrt{\xi}/10,\im z_0>0 \} $ we have $ g(z) \sim g_{sc}(z) $. This shows that for $ 0<t<1 $, we have $ (z_t-z_0) \sim (z_{sc}(t) -z_0) $. The rest of the proof now follows from \cite[Lemma 2.2]{bourgade2018extreme}.
\end{proof}

Furthermore, we have the following lemma regarding the growth of the characteristics, which will be useful for the error estimates in the local relaxation.

\begin{lemma}\label{l.ChracteristcsInt}
For any $ z=E+\i \eta \in \S $, we have
$$ \dfrac{\varphi^4}{N^2} \int_0^t ds \int \dfrac{d \rho(x)}{|z_{t-s} -x|^4 \max(\kappa(x),s^2)} \lesssim \dfrac{\kappa(E)}{\max(\kappa(E),t^2)}. $$
\end{lemma}
\begin{proof}
The proof is essentially the same as in \cite[Lemma A.2]{bourgade2018extreme}, except now we need to consider the Stieltjes transform $ m(z) $ instead of the one for the semicircle law. Recall the relation $ m(z)=z m_{\MP}(z^2) $, and note that for the Stieltjes transform of the Marchenko-Pastur law we have the following results from \cite[Lemma 3.6]{bloemendal2016principal}:
\begin{equation*}
|m_{\MP}(z)| \sim 1,
\ \ \ \ \
\im m_{\MP}(z) \sim \left\{
\begin{aligned}
& \sqrt{\kappa(E)+\eta} & \mbox{if } & E \in [\lambda_-,\lambda_+],\\
&\frac{\eta}{\sqrt{\kappa(E)+\eta}} & \mbox{if } & E \notin [\lambda_-,\lambda_+].
\end{aligned}
\right.
\end{equation*}
The rest of the proof follows from the calculation in \cite[Lemma A.2]{bourgade2018extreme}.
\end{proof}

We now consider the initial value $ f_0 $ on the curve $ \S $. For this purpose we define the set of good trajectories such that the rigidity holds:
\begin{equation*}
\A := \left\{ |x_k^{(\nu)}(t) - \gamma_k|<\varphi^{\frac{1}{2}}N^{-\frac{2}{3}}(\hat{k})^{-\frac{1}{3}} \ \mbox{for all}\ 0 \leq t \leq 1, k \in \llbracket 1,N \rrbracket,0 \leq \nu \leq 1 \right\}.
\end{equation*}
We have the following important estimate for the probability of these events.

\begin{lemma}\label{l.rigidity}
There exists a fixed $ C_0 $ large enough such that the following holds. For any $ D>0 $, there exists $ N_0(D)>0 $ such that for any $ N > N_0 $ we have
$$ \P(\A) > 1-N^{-D}. $$
\end{lemma}

As mentioned in the last section, the rigidity estimates are proved in \cite{pillai2014universality} for fixed $ t $ and $ \nu=0,1 $. The extension to all $ t $ and $ \nu $ is based on the arguments in \cite{bourgade2018extreme,erdHos2015gap}: (1) discretize in $ t $ and $ \nu $; (2) use Weyl's inequality to bound the increments in small time intervals; (3) use the maximum principle to bound the increment in small $ \nu $-intervals.

Conditioned on the rigidity phenomenon, we have the following estimate for the initial conditions. The proof is the same as \cite[Lemma 2.4]{bourgade2018extreme}.

\begin{lemma}\label{l.InitialValue}
In the set $ \A $, for any $ z=E+\i \eta \in \mathscr{R} $, we have $ \im f_0(z) \lesssim \varphi^{1/2} $ if $ \eta > \max(E-\sqrt{\lambda_+},-E+\sqrt{\lambda_-}) $, and $ \im f_0(z) \lesssim \varphi^{1/2} \frac{\eta}{\kappa(z)} $ otherwise. The same bound also holds for $ |\im f_0| $.
\end{lemma}

\subsection{Quantitative relaxation at the edge}\label{s.relaxation}
To prove the edge universality, we first have the following estimate for the size of the observable $ f_t $.

\begin{proposition}\label{p.estimate}
For any (large) $ D>0 $ there exists $ N_0(D) $ such that for any $ N \geq N_0 $ we have
$$ \P\left( \im f_t(z) \lesssim \varphi \dfrac{\kappa(E)^{1/2}}{\max(\kappa(E)^{1/2},t)} \ \mbox{\rm for all} \ 0<t<1 \ \mbox{\rm and} \ z=E+\i\eta \in \S \right) > 1-N^{-D} $$
\end{proposition}
\begin{proof}
For any $ 1 \leq l,m \leq N^{10} $, we define $ t_l:=l N^{-10} $ and
$$ z^{(m)} := E_m + \i\eta_m=E_m+\i\dfrac{\varphi^2}{N \kappa(E_m)^{1/2}}, $$
where
$$ E_m := \inf\left\{ E >0 : \int_{-\infty}^{E^2} \rho_{\MP}(x)\d x \geq \left(m-\frac{1}{2}\right)N^{-10} \right\}. $$
We also define the following stopping times (with respect to $ \mathcal{F}_t = \sigma(B_k(s):0 \leq s \leq t,1 \leq k \leq N) $) which represent the bad events:
$$ \tau_{l,m} := \inf \left\{ 0 \leq s \leq t_l : \im f_s(z_{t-s}^{(m)}) \gtrsim \dfrac{\varphi}{2}\dfrac{\kappa(E_m)^{1/2}}{\max(\kappa(E_m)^{1/2},t_l)}  \right\}, $$
$$ \tau_0 := \inf \left\{ 0 \leq t \leq1: \exists k \in \llbracket -N,N \rrbracket \ \mbox{s.t.} \  |x_k(t)-\gamma_k|>\varphi^{\frac{1}{2}}N^{-\frac{2}{3}}(\hat{k})^{-\frac{1}{3}} \right\}, $$
$$ \tau := \min \left\{ \tau_0,\tau_{l,m}: 0 \leq l,m \leq N^{10}, \kappa(E_m) > \varphi^2 N^{-\frac{2}{3}} \right\}. $$
We also define the convention $ \inf \emptyset =1 $.

We claim that in order to prove the desired result, it suffices to show that for any $ D>0 $ there exists $ \tilde{N}_0(D) $ such that for any $ N \geq \tilde{N}_0(D) $, we have
\begin{equation}\label{e.ProbabilityTau}
\P(\tau=1) > 1-N^{-D}.
\end{equation}

\emph{Step 1.} To see the claim would be enough, we first show the following sets inclusion
\begin{equation}\label{e.SetsInclusion}
\{ \tau=1 \} \bigcap_{\substack{1\leq l,m \leq N^{10} \\ -N \leq k \leq N}} A_{l,m,k} \subset \bigcap_{z \in \S, 0<t<1} \left\{ \im f_t(z) \lesssim \varphi \dfrac{\kappa(E)^{1/2}}{\max(\kappa(E)^{1/2},t)} \right\},
\end{equation}
where
$$ A_{l,m,k} := \left\{ \sup_{t_l \leq u \leq t_{l+1}} \left| \int_{t_l}^u \dfrac{e^{-\frac{s}{2\xi}} \v_k(s) \d B_k(s)}{(z^{(m)} - x_k(s))^2} \right| <N^{-3} \right\}.  $$
To prove this, for any given $ z $ and $ t $, choose $ t_l $ and $ z^{(m)} $ such that $ t_l \leq t <t_{l+1} $ and $ |z-z^{(m)}|<N^{-5} $. Note that by rigidity and the maximum principle (Lemma \ref{l.MaxPrin}) we have $ |\v_k(t)| \lesssim \varphi N^{-2/3} $. Combining this with the definition of $ f_t $, we have $ |f_t(z) - f_t(z^{(m)})| < N^{-2} $. Moreover, note that we have the following estimates
\begin{multline*}
\sup_{t_l \leq u \leq t_{l+1}} \left| \int_{t_l}^u \dfrac{e^{-\frac{t}{2\xi}}}{2N}\sum_{-N \leq k \leq N} \dfrac{\v_k}{(x_k-z)^2(x_k+z)} \d t \right|\\
\lesssim N^{-10}  \dfrac{1}{2N} \varphi N^{-2/3} \max_{t_l \leq t \leq t_{l+1}} \sum_{-N \leq k \leq N} \left| \dfrac{1}{(x_k-z)^2(x_k+z)} \right| < N^{-5},
\end{multline*}
and
\begin{multline*}
\sup_{t_l \leq u \leq t_{l+1}} \left| \int_{t_l}^u \left( 1 - \dfrac{1}{\xi} \right) e^{-\frac{t}{2\xi}} \sum_{-N \leq k \leq N} \dfrac{z \v_k}{x_k^2(x_k -z) (x_k+z)} \d t \right|\\
\lesssim N^{-10} \varphi N^{-2/3} \max_{t_l \leq t \leq t_{l+1}} \sum_{-N \leq k \leq N} \left| \dfrac{1}{(x_k -z) (x_k+z)} \right| < N^{-5},
\end{multline*}
\begin{multline*}
\sup_{t_l \leq u \leq t_{l+1}} \left| \int_{t_l}^u \left( 1 - \dfrac{1}{\xi} \right) e^{-\frac{t}{2\xi}} \sum_{-N \leq k \leq N} \dfrac{z^3 \v_k}{ x_k^2 (x_k -z)^2 (x_k+z)^2} \d t \right|\\
\lesssim N^{-10} \varphi N^{-2/3} \max_{t_l \leq t \leq t_{l+1}} \sum_{-N \leq k \leq N} \left| \dfrac{1}{(x_k -z)^2 (x_k+z)^2} \right| < N^{-5}.
\end{multline*}
Similarly, we also have
$$ \sup_{t_l \leq u \leq t_{l+1}} \left| \int_{t_l}^u  \left(s_t(z)+\dfrac{z}{2\xi}\right)(\partial_z f_t)  \d t \right| < N^{-5},\ \ \ \  \sup_{t_l \leq u \leq t_{l+1}} \left| \int_{t_l}^u  \dfrac{1}{4N}(\partial_{zz}f_t)  \d t \right| < N^{-5}. $$
Then based on the dynamics \eqref{e.dynamics}, under the event $ \cap_k A_{l,m,k} $, we have $ |f_t(z^{(m)}) - f_{t_l}(z^{(m)})| < N^{-2} $. This completes the proof for \eqref{e.SetsInclusion}.

\emph{Step 2.} To prove the claim, we also need to estimate the probability for the events $ A_{l,m,k} $. To do this, we use the following Burkholder-Davis-Gundy type inequality (see e.g. \cite[Appendix B.6]{shorack2009empirical}). For some fixed constants $ c>0 $ and  $ \alpha>0 $, for any martingale $ M $ we have
\begin{equation}\label{e.BDG}
\P \left( \sup_{0 \leq u \leq t} |M_u| \geq \alpha \la M \ra _t^{\frac{1}{2}} \right) \leq e^{-c \alpha^2}.
\end{equation}
Note that we have the deterministic bound $ \int_{t_l}^u \frac{|\v_k (s)|^2 \d s}{|z^{(m)} - x_k(s)|^4} \ll N^{-6} $. By taking $ \alpha = \varphi^{1/10} $, this implies $ \P(A_{l,m,k}) \geq 1 - e^{-c \varphi^{1/5}} $, and then a union bound yields
$$ \P \left( \bigcap_{1\leq l,m \leq N^{10}, -N \leq k \leq N} A_{l,m,k} \right) > 1-N^{-D}. $$
Together with the sets inclusion \eqref{e.SetsInclusion}, this concludes that \eqref{e.ProbabilityTau} implies the desired result.

\emph{Step 3.} It remains to prove \eqref{e.ProbabilityTau}. For simplicity of notations, let $ t=t_l $, $ z=E+\i\eta=z^{(m)} $ for some arbitrary fixed $ 1 \leq l,m \leq N^{10} $. Consider the function $ g_u(z) := f_u(z_{t-u}) $. By Lemma \ref{l.characteristics} and Lemma \ref{l.InitialValue}, we have $ \im g_0(z) \lesssim \frac{\varphi}{10}\frac{\kappa(E_m)^{1/2}}{\max(\kappa(E_m)^{1/2},t)} $. Therefore we only need to bound the increments of $ g $. Using Lemma \ref{l.dynamics}, by the It\^{o}'s formula we know it satisfies the following stochastic differential equation
\begin{equation}
d g_{u \wedge \tau}(z) = \ep_u(z_{t-u})\d (u \wedge \tau) - \dfrac{e^{-\frac{u}{2\xi}}}{\sqrt{N}}\sum_{-N \leq k \leq N}\dfrac{\v_k(u)}{(z_{t-u}-x_k(u))^2}\d B_k(u \wedge \tau),
\end{equation}
where
\begin{align*}
\ep_u(z) &:= (s_u(z)-m(z))\partial_z f_u + \dfrac{1}{4N} (\partial_{zz} f_u) + \dfrac{e^{-\frac{u}{2\xi}}}{2N} \sum_{-N \leq k \leq N} \dfrac{\v_k(u)}{(x_k -z)^2 (x_k+z)}\\
&\quad  + \left( 1- \dfrac{1}{\xi} \right) e^{-\frac{t}{2\xi}} \left( \sum_{-N \leq k \leq N} \dfrac{3z \v_k}{2 x_k^2 (x_k -z)(x_k+z)} + \sum_{-N \leq k \leq N} \dfrac{z^3 \v_k}{x_k^2 (x_k-z)^2 (x_k+z)^2} \right).
\end{align*}
In this step, we aim to estimate the first term $ \sup_{0 \leq s \leq t}|\int_0^s \ep_u(z_{t-u}) \d (u \wedge \tau)| $. First, we have
\begin{multline*}
\int_0^t \left| \left( s_u(z_{t-u}) -m(z_{t-u}) \right) \partial_z f (z_{t-u}) \right|\d (u \wedge \tau)\\ \lesssim \int_0^t \dfrac{\varphi}{N \im (z_{t-u})} \sum_{-N \leq k \leq N} \dfrac{|\v_k(u)|}{|z_{t-u} -x_k(u)|^2} \d (u \wedge \tau)
\lesssim \int_0^t \dfrac{\varphi \im f_u(z_{t-u})}{N \left( \im (z_{t-u}) \right)^2}\d (u \wedge \tau) \\
\lesssim \int_0^t \dfrac{\varphi^2 du}{N (\eta + (t-u)\kappa(z)^{1/2})^2}\dfrac{\kappa(E)^{1/2}}{\max(\kappa(E)^{1/2},t)} \lesssim \dfrac{\kappa(E)^{1/2}}{\max \left( \kappa(E)^{1/2},t \right)}.
\end{multline*}
To bound the $ |s_u - m| $ term above, we have used the local law for singular values \eqref{e.LocalLaw} simultaneously for all $ 0 \leq u \leq t $ (which is similar to Lemma \ref{l.rigidity}). The last two inequalities follow from Lemma \ref{l.characteristics}. We also have
\begin{equation*}
\sup_{0 \leq s \leq t} \left| \int_0^s \dfrac{1}{4N}(\partial_{zz} f_u(z_{t-u})) \d (u \wedge \tau) \right| \lesssim \int_0^t \dfrac{\im f_u(z_{t-u})}{N \left( \im (z_{t-u}) \right)^2}\d (u \wedge \tau)  \lesssim \dfrac{\kappa(E)^{1/2}}{\varphi \max(\kappa(E)^{1/2},t)},
\end{equation*}
\begin{multline*}
\sup_{0 \leq s \leq t} \left| \int_0^s \dfrac{e^{-\frac{u}{2\xi}}}{2N} \sum_{-N \leq k \leq N} \dfrac{\v_k(u)}{(x_k - z_{t-u})^2(x_k+z_{t-u})} \d (u \wedge \tau) \right|\\
\lesssim \int_0^t  \dfrac{\im f_u(z_{t-u})}{N \left( \im (z_{t-u}) \right)}\d (u \wedge \tau) \lesssim \dfrac{\kappa(E)^{1/2}}{\varphi \max(\kappa(E)^{1/2},t)},
\end{multline*}
And similarly,
\begin{multline*}
\sup_{0 \leq s \leq t} \left| \left( 1-\dfrac{1}{\xi} \right)e^{-\frac{u}{2\xi}} \int_0^s \sum_{-N \leq k \leq N} \dfrac{z_{t-u} \v_k(u)}{x_k^2(x_k(u) - z_{t-u}) (x_k(u) + z_{t-u})} \d (u \wedge \tau) \right|\\
\lesssim \int_0^t \im f_u(z_{t-u}) \d (u \wedge \tau)  \lesssim \dfrac{\varphi}{2}\dfrac{\kappa(E)^{1/2}}{\max \left( \kappa(E)^{1/2},t \right)},
\end{multline*}
\begin{multline*}
\sup_{0 \leq s \leq t} \left| \left( 1-\dfrac{1}{\xi} \right)e^{-\frac{u}{2\xi}} \int_0^s \sum_{-N \leq k \leq N} \dfrac{z_{t-u}^3 \v_k(u)}{x_k^2(x_k(u) - z_{t-u})^2 (x_k(u) + z_{t-u})^2} \d (u \wedge \tau) \right|\\
\lesssim \dfrac{\varphi}{2}\dfrac{\kappa(E)^{1/2}}{\max \left( \kappa(E)^{1/2},t \right)}.
\end{multline*}

\emph{Step 4.} Finally we focus on the estimate for $ \sup_{0 \leq s \leq t}|M_s| $ where
$$ M_s := \int_0^s \dfrac{e^{-\frac{u}{2\xi}}}{\sqrt{N}} \sum_{-N \leq k \leq N} \dfrac{\v_k(u)}{(z_{t-u} - x_k(u))^2} \d B_k(u \wedge \tau). $$
Note that for all $ k $ and $ u < \tau $ we have $ |z_{t-u} - \gamma_k| \lesssim |z_{t-u} - x_k(u)| $ due to the fact $ |x_k(u) - \gamma_k| \ll |z_{t-u} - \gamma_k| $. Using \eqref{e.BDG} again we have
$$ \sup_{0 \leq s \leq t}|M_s|^2 \lesssim \varphi^{\frac{1}{10}} \int_0^t \dfrac{1}{N}\sum_{-N \leq k \leq N} \dfrac{\v_k(u)^2}{|z_{t-u} - \gamma_k|^4} \d (u \wedge \tau) $$
with overwhelming probability. By a similar argument in \cite[equation (2.17)]{bourgade2018extreme} and Lemma \ref{l.ChracteristcsInt}, we conclude
$$ \sup_{0 \leq s \leq t} |M_s|^2 \lesssim \varphi^{1/5} \dfrac{\kappa(E)}{\max(\kappa(E),t^2)}. $$

Hence, based on the previous estimates and a union bound we have proved that for any $ D>0 $ there exists $ N_0 $ such that for every $ N > N_0 $ we have
$$ \P \left( \sup_{0 \leq l,m \leq N^{10},\kappa(E_m)>\varphi^2 N^{-2/3},0 \leq s \leq t_l} \im f_{s \wedge \tau}(z_{t_l - s \wedge \tau}^{(m)}) \lesssim \dfrac{\varphi}{2}\dfrac{\kappa(E_m)^{1/2}}{\max(\kappa(E_m)^{1/2},t_l)} \right) \geq 1- N^{-D}. $$
Together with Lemma \ref{l.InitialValue}, we have proved \eqref{e.ProbabilityTau}, which then completes the proof.
\end{proof}

Based on the previous estimate on $ \im f_t $, we now state the quantitative relaxation of the singular values dynamics at the edge. Remember that $ \{s_k\} $ and $ \{r_k\} $ satisfy the same equation \eqref{e.DBM}, with respective initial conditions that correspond to a general sample covariance matrix and the Wishart ensemble.

\begin{theorem}\label{t.Relaxation}
For any $ D>0 $ and $ \ep>0 $ there exists $ N_0>0 $ such that for any $ N > N_0 $ we have
$$ \P \left( |s_k(t) - r_k(t)| \lesssim \dfrac{N^\ep}{Nt} \ \mbox{for all} \ k \in \llbracket 1,N \rrbracket \ \mbox{in} \ t \in [0,1] \right) > 1-N^{-D}. $$
\end{theorem}
\begin{proof}
Let $ z=\gamma_k + \i \frac{\varphi^2}{N \sqrt{\kappa(\gamma_k)}} \in \S $, then conditioned on the rigidity phenomenon $ \A $, by the nonnegativity of $ \v_k $'s shown in Lemma \ref{l.MaxPrin} we have
$$ |\v_k(t)| \lesssim \dfrac{\varphi^2}{N \sqrt{\kappa(\gamma_k)}} \im f_t(z). $$
By the arguments in \cite[Corollary 2.7]{bourgade2018extreme}, based on Proposition \ref{p.estimate} we conclude
$$ \P \left( |\v_k^{(\nu)}(t)| \lesssim \dfrac{\varphi^{10}}{N} \dfrac{1}{\max((\hat{k}/N)^{1/3},t)} \ \mbox{for all} \ k \in \llbracket 1,N \rrbracket \ \mbox{in} \ t \in [0,1] \right) > 1-N^{-D}. $$
Again by Lemma \ref{l.MaxPrin}, we know $ -\v_k \leq \u_k \leq \v_k $. Therefore we have
$$ \P \left( |\u_k^{(\nu)}(t)| \lesssim \dfrac{\varphi^{10}}{N} \dfrac{1}{\max((\hat{k}/N)^{1/3},t)} \ \mbox{for all} \ k \in \llbracket 1,N \rrbracket \ \mbox{in} \ t \in [0,1] \right) > 1-N^{-D}. $$
The result then follows as the proof in \cite[Theorem 2.8]{bourgade2018extreme}.
\end{proof}

\subsection{Proof for $ \xi=1 $}\label{s.CaseOne}
Due to the fact that $ \rho(x) $ is the semicircle law in this special case, now the advection equation \eqref{e.AdvectionPDE} is the same as \cite[equation (1.12)]{bourgade2018extreme}, whose characteristics has an explicit formula. This coincidence makes it easy to adapt the previous proofs for $ \xi \neq 1 $ to this case. With a little abuse of notations, define the curve
$$ \S:= \left\{ z=E+\i y: 0 <E< 2 - \varphi^2 N^{-2/3},\ y=\varphi^2/\left( N \kappa(E)^{1/2} \right) \right\}, $$
where $ \kappa(z) := |z-2| $. Under the framework of such notations and the rigidity estimates \eqref{e.RigidityOneSoft} and \eqref{e.RigidityOneHard}, by the arguments in \cite[Section 2]{bourgade2018extreme}, all previous results in Section \ref{s.geometry} still hold for $ \xi=1 $. Then using the same method we can prove Proposition \ref{p.estimate} with few changes, and consequently the proof for Theorem \ref{t.Relaxation} is completed.

\section{Rate of Convergence to Tracy-Widom Law}\label{s.Comparison}
\subsection{Quantitative Green function comparison}
Following the general three-step strategy in the dynamical approach, the derivation of the rate of convergence relies on both the relaxation and the Green function comparison theorem from \cite{GreenFunctionComparison}. In the context of sample covariance matrices, this Lindeberg exchange strategy based on the fourth moments matching condition was first used by Tao and Vu in \cite{TaoVu}. To obtain an explicit convergence rate, we need a quantitative version of the comparison theorem.

For the statement, we consider a fixed $ |E-\lambda_+| < \varphi N^{-2/3} $, a scale $ \rho=\rho(N) \in [N^{-1},N^{-2/3}] $, and a function $ f=f(N):\R \to \R $ satisfying
$$ \|f^{(k)} \|_{L^\infty([E,E+\rho])} \leq C_k \rho^{-k},\ \ \ \|f^{(k)} \|_{L^{\infty}([E^+,E^++1])} = O(1),\ \ \ 0 \leq k \leq 2. $$
where $ E^+ = E+\varphi N^{-2/3} $. We assume that $ f $ is non-decreasing on $ (-\infty,E^+] $,  $ f(x) \equiv 0 $ for $ x<E $ and $ f(x) \equiv 1 $ for $ E+\rho<x\leq E^+ $; and also assume $ f $ is non-increasing on $ [E^+,\infty) $, $ f \equiv 0 $ for $ x>E^+ +1 $. Furthermore, let $ F $ be a fixed smooth non-increasing function such that $ F(x) \equiv 1 $ for $ x \leq 0 $ and $ F(x) \equiv 0 $ for $ x \geq 1 $.

\begin{theorem}[Quantitative Green function comparison]\label{t.GreenCompare}
There exists $ C>0 $ such that the following holds. Let $ \Xv, \Xw $ be data matrices satisfying assumptions \eqref{e.Assumption1} and \eqref{e.Assumption2}, and $ \Hv,\Hw $ be the corresponding sample covariance matrices. Assume that the first three moments of the entries are the same, i.e. for all $ 1 \leq i \leq M $,  $ 1 \leq j \leq N $ and $ 1 \leq k \leq 3 $ we have
$$ \Ev (x_{ij}^k) = \Ew (x_{ij}^k). $$
Assume also that for some parameter $ t=t(N) $ we have
$$ \left| \Ev (\sqrt{M}x_{ij})^4 - \Ew (\sqrt{M}x_{ij})^4 \right| \leq t. $$
With the above notations for the test functions $ f $ and $ F $, we have
$$ \left| (\Ev - \Ew) F \left( \tr f(H) \right) \right| \leq \varphi^C \left( \dfrac{1}{N^{18} \rho^{20}} + \dfrac{t}{N \rho} + \dfrac{1}{(N \rho)^2} + \dfrac{1}{N^2} \right). $$
\end{theorem}
\begin{proof}
We follow the notations in \cite{pillai2014universality} and the reasoning from \cite[Theorem 17.4]{erdos2017dynamical}. Fix a bijective ordering map on the index set of the independent matrix elements,
$ \phi:\{(i,j):1 \leq i \leq M, 1 \leq j \leq N\} \to \{1,\cdots,MN\} $
and define the family of random matrices $ X_\gamma $, $ 0 \leq \gamma \leq MN $
\begin{equation*}
[X_\gamma]_{ij} = \left\{
\begin{aligned}
& [\Xv]_{ij} & \mbox{if} & \ \  \phi(i,j)>\gamma,\\
& [\Xw]_{ij} & \mbox{if} & \ \  \phi(i,j) \leq \gamma.
\end{aligned}
\right.
\end{equation*}
Note that in particular we have $ X_0 = \Xv $ and $ X_{MN} = \Xw $. Denote sample covariance matrices $ H_\gamma $ as
$$ H_\gamma := X_\gamma^* X_\gamma. $$

Let $ \chi $ be a fixed, smooth, symmetric cutoff function such that $ \chi(x)=1 $ if $ |x|<1 $ and $ \chi(x)=0 $ if $ |x|>2 $. By the Helffer-Sj\"{o}strand formula, if $ \lambda_i $'s are the (real) eigenvalues of a matrix $ H $, we have
$$ \sum f(\lambda_i) = \int_{\C} g(z) \tr \dfrac{1}{H-z} \d m(z), $$
where $ \d m $ is the Lebesgue measure on $ \C $, and the function $ g $ is defined as
$$ g(z) := \dfrac{1}{\pi} \left( \i y f''(y) \chi(y) + \i (f(x) + \i y f'(x))\chi'(y) \right),\ \ \ z=x+\i y. $$
Define
$$ \Xi^H := \int_{|y|>N^{-1}} g(z) \tr (H-z)^{-1} \d m(z), $$
and we have the bound (see \cite[Section 5.2]{bourgade2018extreme})
$$ \left|\sum f(\lambda_i) - \Xi^H \right| \leq O \left( \dfrac{\varphi^C}{(N \rho)^2} \right). $$
This shows that it suffices to show
\begin{equation}\label{e.TelescopicSum}
\left| \E F(\Xi^{H_\gamma}) - \E F(\Xi^{H_{\gamma-1}}) \right| \leq\dfrac{ \varphi^C }{N^2} \left( \dfrac{1}{N^{18} \rho^{20}} + \dfrac{t}{N \rho} + \dfrac{1}{(N \rho)^2} + \dfrac{1}{N^2} \right).
\end{equation}
For an arbitrarily fixed $ \gamma $ corresponding to $ (i,j) $, we can write
$$ X_{\gamma-1}=Q+V,\ \ \ V:=\Xv_{ij}E^{(ij)},\ \ \ \ X_\gamma = Q+W,\ \ \ W:= \Xw_{ij}E^{(ij)}.  $$
where $ Q $ coincides with $ X_{\gamma-1} $ and $ X_\gamma $ except on the $ (i,j) $ position (where it is 0). We define the Green functions
$$ R:= (Q^*Q-z)^{-1},\ \ \ S:=(H_{\gamma-1}-z)^{-1}. $$
By Taylor expansion, for some fixed order $ m $, we have
\begin{multline}\label{e.Taylor}
\E F(\Xi^{H_{\gamma}}) - \E F(\Xi^{H_{\gamma-1}}) = \sum_{l=1}^{m-1} \E \dfrac{F^{(l)}(\Xi^Q)}{l !}\left( (\Xi^{H_\gamma} - \Xi^Q)^l - (\Xi^{H_{\gamma-1}} - \Xi^Q)^l \right)\\
 + O \left( \|F^{(m)} \|_\infty \right) \left( \E \left( (\Xi^{H_\gamma} - \Xi^Q)^m + (\Xi^{H_{\gamma-1}} - \Xi^Q)^m \right) \right).
\end{multline}

First we estimate the $ m $-th order error term. By the first order resolvent expansion we have
\begin{align*}
|\Xi^{H_{\gamma}} - \Xi^Q| &\leq \int_{|y|>N^{-1},|x|<\lambda_+ +2} |g(z)|\left| \tr R(z)(V^*Q+Q^*V+V^*V)S(z) \right| \d m(z) \\
&\leq \varphi^C N \int_{|y|>N^{-1},|x|<\lambda_+ +2} |g(z)| \|S(z) \|_\infty \|R(z) \|_\infty \d m(z)
\end{align*}
with overwhelming probability, where we use the fact that there are only $ O(N) $ nonzero entries with size $ O(N^{-1}) $ in the matrix $ V^*Q+Q^*V+V^*V $. By the strong local Marchenko-Pastur law (\cite[Theorem 3.1]{pillai2014universality}), for any $ D>0 $ we have
$$ \P \left( \max_{j}|S_{jj}(z) - m_{\MP}(z)| + \max_{j \neq k}|S_{jk}(z)| \leq \varphi^C \left( \dfrac{1}{Ny}+\sqrt{\dfrac{\im m_{\MP}(y)}{Ny}} \right) \right) > 1- N^{-D}. $$
Same bound for $ \|R(z) \|_\infty $ also holds (see \cite[Lemma 5.4]{pillai2014universality}). This shows that
$$ \E  (\Xi^{H_\gamma} - \Xi^Q)^m = O \left( \varphi^C / (N^m \rho^m) \right),\ \ \ \E  (\Xi^{H_{\gamma-1}} - \Xi^Q)^m = O \left( \varphi^C / (N^m \rho^m) \right). $$
Therefore the $ m $-th order term in \eqref{e.Taylor} can be bounded by $ \varphi^C N^{-2}(N^{-m+2} \rho^{-m}) $.

Next we consider the first order term in the Taylor expansion. By the resolvent expansion, we have
$$ S=R-R \Av R+(R \Av)^2R-(R \Av)^3R+ \cdots - (R \Av)^{11}R + (R \Av)^{12}S, $$
where
$$ \Av = V^*Q + Q^*V + V^*V. $$
Denote
$$ \Rv^{(n)} := (-1)^{n} \tr (R \Av)^n R,\ \ \ \ \Omv := \tr (R \Av)^{12} S. $$
Then we have
$$ \E F'(\Xi^Q) \left( \Xi^{H_{\gamma-1}} - \Xi^{H_{\gamma}} \right) =  \E F'(\Xi^Q) \int g(z) \left( \sum_{n=1}^{11} \left( \Rv^{(n)} - \Rw^{(n)} \right) +(\Omv - \Omw) \right) \d m(z). $$
Since the first three moments of the two matrices are identical, we know that the case $ n=1 $ gives null contribution.

For $ n=2 $, note that the entries of the matrix $ A $ satisfy the following relation
\begin{equation*}
A_{ab}=\left\{
\begin{aligned}
& x_{ij} x_{ib} & \quad & \mbox{if} \  a=j,b \neq j,\\
& x_{ij} x_{ia} & \quad & \mbox{if} \  a \neq j, b=j,\\
& 0 & \quad &\mbox{otherwise}.
\end{aligned}
\right.
\end{equation*}
This shows that
$$ \E (\Rv^{(2)} - \Rw^{(2)}) \leq N \left( \dfrac{t}{N^2} \right) \left(\max_{i \neq j} |R_{ij}|\right)^2 \left(\max_i |R_{ii}|\right), $$
where we used that in the expansion
\begin{align*}
\tr (R \Av)^2R & = \sum_{k} \sum_{a_1,b_1,a_2,b_2} R_{k a_1}\Av_{a_1 b_1}R_{b_1 a_2}\Av_{a_2 b_2}R_{b_2 k}\\
& = \sum_{k} \sum_{(a_1,b_1,a_2,b_2) \neq (j,j,j,j)} R_{k a_1}\Av_{a_1 b_1}R_{b_1 a_2}\Av_{a_2 b_2}R_{b_2 k}  + \sum_k  R_{kj}\Av_{jj}R_{jj}\Av_{jj}R_{jk}
\end{align*}
due to the moment matching condition, the terms that make nontrivial contribution are only in the second summation, which is
$$ \sum_k R_{kj}R_{jj}R_{jk}(\Xv_{ij})^4. $$
Here we also use the fact that the contribution for the terms with $ k $ equal $ i $ or $ j $ is combinatorially negligible. By the local law, we conclude
$$ \E F'(\Xi^Q) \int g(z) \left( \Rv^{(2)} - \Rw^{(2)} \right) \d m(z) = O \left( \dfrac{\varphi^C t}{N} \right) \int \dfrac{|g(z)|}{(Ny)^2} \d m(z) = O \left( \dfrac{\varphi^C}{N^2} \dfrac{t}{N\rho}  \right). $$
For the terms $ n=3,\dots,11 $, as explained in \cite[Lemma 5.4]{pillai2014universality}, their contributions are of smaller order.
Similarly, for the term $ (\Omv - \Omw) $, as shown in \cite[Lemma 5.4]{pillai2014universality} we have $ \Omv = O(N^{-4}) $. Therefore we have
$$ \E F'(\Xi^Q) \int g(z) \left( (\Omv - \Omw) \right) \d m(z) = O \left( \dfrac{\varphi^C}{N^2} \dfrac{1}{N^2 \rho} \right)=O \left( \dfrac{\varphi^C}{N^2} \right) \left( \dfrac{1}{(N \rho)^2} + \dfrac{1}{N^2} \right). $$

Moreover, as explained in \cite[Theorem 17.4]{erdos2017dynamical}, the contributions of higher order terms in Taylor expansion are of smaller order. Combining the estimates and taking $ m=20 $ (we will see the reason in the next section) gives us \eqref{e.TelescopicSum}. Finally, a telescopic summation yields the desired result.
\end{proof}

\subsection{Proof of Theorem \ref{t.Rate}}
Let $ s \in \R $. If $ |s|>\varphi $, due to the rigidity we know that for any $ D>0 $ and large enough $ N $, we have $ \P\left( N^{2/3}(\lambda_N - \lambda_+) \leq s \right) = \P(\TW  \leq s) + O(N^{-D}) $. So in the following discussion we assume $ |s| \leq \varphi $.

Denoting a non-decreasing function $ f_1 $ such that $ f_1(x)=1 $ for $ x>\lambda_+ + sN^{-2/3} $ and $ f_1(x)=0 $ for $ x<\lambda_+ + sN^{-2/3} - \rho $. We also define $ f_2(x) := f_1(x-\rho) $. Then we have
\begin{equation}\label{e.ProbEst}
\E_H F \left( \sum_{i=1}^N f_1(\lambda_i)  \right) \leq \P_H \left( \lambda_N < \lambda_+ + s N^{-2/3} \right) \leq \E_H F \left( \sum_{i=1}^N f_2(\lambda_i)  \right).
\end{equation}

Moreover, as discussed in \cite{erdos2011universality,pillai2014universality}, we can find an $ M \times N $ matrix $ \tilde{X}_0 $ such that the Gaussian divisible ensemble $ \tilde{X}_t := e^{-t/2}\tilde{X}_0 + (1-e^{-t})^{1/2}X_G $, where $ X_G $ is a matrix whose entries are independent Gaussian random variables with mean 0 and variance 1, satisfies the following: for $ 1 \leq k \leq 3 $,
$$ \E (\sqrt{M}X_{ij})^k = \E [\tilde{X}_t]_{ij}^k,\ \ \ \ |\E (\sqrt{M}X_{ij})^4 - \E [\tilde{X}_t]_{ij}^4| \lesssim t. $$
By the quantitative Green function comparison theorem \ref{t.GreenCompare}, we obtain the following bound
\begin{multline*}
\E_{\tilde{X}_t} F \left( \sum_{i=1}^N f_1(\lambda_i) \right) - \varphi^C \left( \dfrac{1}{N^{18} \rho^{20}} + \dfrac{t}{N \rho} + \dfrac{1}{(N \rho)^2} + \dfrac{1}{N^2} \right)
\leq \P_H \left( \lambda_N < \lambda_+ + s N^{-2/3} \right) \\
\leq \E_{\tilde{X}_t} F \left( \sum_{i=1}^N f_2(\lambda_i) \right) + \varphi^C \left( \dfrac{1}{N^{18} \rho^{20}} + \dfrac{t}{N \rho} + \dfrac{1}{(N \rho)^2} + \dfrac{1}{N^2} \right).
\end{multline*}
Using \eqref{e.ProbEst} for $ \tilde{X}_t $, the estimate becomes
\begin{multline*}
\P_{\tilde{X}_t} \left( \lambda_N < \lambda_+ + s N^{-2/3} -\rho \right)  - \varphi^C \left( \dfrac{1}{N^{18} \rho^{20}} + \dfrac{t}{N \rho} + \dfrac{1}{(N \rho)^2} + \dfrac{1}{N^2} \right)\\
\leq \P_H \left( \lambda_N < \lambda_+ + s N^{-2/3} \right) \leq \\
\P_{\tilde{X}_t} \left( \lambda_N < \lambda_+ + s N^{-2/3} +\rho \right)  + \varphi^C \left( \dfrac{1}{N^{18} \rho^{20}} + \dfrac{t}{N \rho} + \dfrac{1}{(N \rho)^2} + \dfrac{1}{N^2} \right).
\end{multline*}
After combined with the edge relaxation Theorem \ref{t.Relaxation}, the estimate now gives us
\begin{multline*}
\P_{\WS} \left( N^{2/3}(\lambda_N - \lambda_+) < s -N^{2/3}\rho - \dfrac{N^\ep}{N^{1/3}t} \right) - \varphi^C \left( \dfrac{1}{N^{18} \rho^{20}} + \dfrac{t}{N \rho} + \dfrac{1}{(N \rho)^2} + \dfrac{1}{N^2} \right)\\
 \leq \P_H \left( N^{2/3}(\lambda_N - \lambda_+) < s \right) \leq \\
\P_{\WS} \left( N^{2/3}(\lambda_N - \lambda_+) < s +N^{2/3}\rho + \dfrac{N^\ep}{N^{1/3}t} \right) + \varphi^C \left( \dfrac{1}{N^{18} \rho^{20}} + \dfrac{t}{N \rho} + \dfrac{1}{(N \rho)^2} + \dfrac{1}{N^2} \right).
\end{multline*}
Moreover, as shown in \cite{el2006rate,ma2012accuracy}, we know
$$ \P_{\WS} \left( N^{2/3}(\lambda_N - \lambda_+) <s \right)  = \P(\TW <s)+ O(N^{-2/3}). $$
By using this Wishart result and the boundedness of the density for $ \TW $, we obtain
\begin{multline}\label{e.RateNullCase}
\P_H \left(N^{2/3}(\lambda_N - \lambda_+) <s \right) - \P \left( \TW <s \right)\\
 = O \left( N^\ep \right) \left( N^{2/3}\rho + \dfrac{1}{N^{1/3}t} + \dfrac{1}{N^{18} \rho^{20}} + \dfrac{t}{N \rho} + \dfrac{1}{(N \rho)^2} + \dfrac{1}{N^{2/3}} \right).
\end{multline}
The optimal bound $ N^{-2/9 + \ep} $ is obtained for $ t=N^{-1/9} $ and $ \rho = N^{-8/9} $. This completes the whole proof for Theorem \ref{t.Rate}.

\section{Generalization to General Population Matrices}\label{s.GeneralPopulation}
In this section, we proceed to generalize our previous results for sample covariance matrices of type $ X^*X $ (which corresponds to the identity population) and aim to derive the rate of convergence to the Tracy-Widom distribution for the (rescaled) largest eigenvalue of separable sample covariance matrices with general population. Throughout this section, we will follow the notations and the setup in the the work by Lee and Schnelli \cite{lee2016tracy}.

Let $ X=(x_{ij}) $ be defined as in \eqref{e.Assumption1} and \eqref{e.Assumption2}. For some deterministic $ M \times M $ matrix $ T $, the sample covariance matrices associated with data matrix $ X $ and population matrix $ \Sigma:=T^*T $ is defined as $ \mathcal{Q}:=(TX)(TX)^* $. Note that the $ M \times M $ matrix $ \mathcal{Q} $ and the matrix
$$ Q:=X^*\Sigma X $$
share the same non-trivial eigenvalues. Since we are studying the largest eigenvalue of the sample covariance matrix, it is more convenient to work with the matrix $ Q $ (called the separable sample covariance matrix) for some technical reasons. We denote the eigenvalues of $ Q $ in increasing order by $ \mu_1 \leq \cdots \leq \mu_N $.

As mentioned previously, for the null case (i.e. the population matrix is identity), it is well known that the empirical eigenvalue distribution of a sample covariance matrix converges weakly in probability to the Marchenko-Pastur law. Under the general setting, however, this results need to be modified and the limiting measure (called the deformed Marchenko-Pastur law) will depends on the spectrum of the population matrix. Let $ \sigma_1 \leq \cdots \leq \sigma_M $ be the eigenvalues of the population matrix $ \Sigma $, we denote by $ \hat{\rho}=\hat{\rho}(M) $ the empirical eigenvalue distribution of $ \Sigma $, which is defined as
$$ \hat{\rho}:=\dfrac{1}{M}\sum_{j=1}^M \delta_{\sigma_j}. $$
The deformed Marchenko-Pastur law $ \hat{\rho}_{\fc} $ is defined in the following way. The Stieltjes transform $ \hat{m}_{\fc} $ of the probability measure is given by the unique solution of the equation
$$  \hat{m}_{\fc}(z)=\dfrac{1}{-z+\xi^{-1}\int \frac{1}{t  \hat{m}_{\fc}(z) + 1}\d \hat{\rho}(t)},\ \ \ \im \hat{m}_{\fc}(z) \geq 0,\ \ \ z \in \C^+. $$
It has been discussed in \cite{knowles2017anisotropic} that $  \hat{m}_{\fc} $ is associated to a continuous probability density $  \hat{\rho}_{\fc} $ with compact support in $ [0,\infty) $. Moreover, the density $\hat{\rho}_{\fc}  $ can be obtained from $ \hat{m}_{\fc} $ via the Stieltjes inversion formula
$$ \hat{\rho}_{\fc}(E) = \lim_{\eta \downarrow 0}\dfrac{1}{\pi}\im \hat{m}_{\fc}(E+\i \eta). $$
The typical location of the largest eigenvalue, which is the rightmost endpoint of the support of the density $ \hat{\rho}_{\fc} $ is determined in the following way. Recall that $ \xi:=N/M $, we define $ \xi_+ $ as the largest solution of the equation
$$ \int \left( \dfrac{t \xi_+}{1-t\xi_+} \right)^2 \d \hat{\rho}_{\fc}(t) = \xi. $$
We remark that $ \xi_+ $ is unique and $ \xi_+ \in [0,\sigma_M^{-1}] $. We then introduce the typical location for the largest eigenvalue $ E_+ $ by
\begin{equation}\label{e.RightEndpoint}
E_+ := \dfrac{1}{\xi_+}\left( 1+\xi^{-1}\int \dfrac{t \xi_+}{1-t\xi_+} \d \hat{\rho}_{\fc}(t) \right).
\end{equation}

Now we state our assumptions on the population matrix $ \Sigma $ that are needed to prove the explicit rate of convergence. For general random matrices $ X $ we require $ \Sigma $ to be diagonal, and we will show later this diagonal condition can be removed if $ X $ is Gaussian. We further need the following assumption for the spectrum of the population matrix $ \Sigma $. Throughout this section, we assume the following:
\begin{equation}\label{e.AssumptionPopulation}
\liminf_M \sigma_1 >0,\ \ \ \limsup_M \sigma_M <\infty,\ \ \ \mbox{and}\ \ \ \limsup_M \sigma_M \xi_+ <1.
\end{equation}
\begin{remark}
The assumption \eqref{e.AssumptionPopulation} is the same as \cite[Assumption 2.2]{lee2016tracy}. It is first used in \cite{bao2015universality,knowles2017anisotropic} to prove the local deformed Marchenko-Pastur law. In particular, the last inequality ensures that the density $ \hat{\rho}_{\fc} $ exhibits a square-root behavior near the right edge of its support, which is crucial to derive the local law.
\end{remark}

It is natural to note that with a general population matrix, the distribution of the largest eigenvalue should not behave exactly like the null case. Besides the typical location of the largest eigenvalue is changed, the normalization constant of the fluctuation is also different. Therefore, we introduce the following normalization constant $ \gamma_0 $ given by
\begin{equation}\label{e.ScalingConstant}
\dfrac{1}{\gamma_0^3} = \dfrac{1}{\xi}\int \left( \dfrac{t}{1-t\xi_+} \right)\d \hat{\rho}(t) + \dfrac{1}{\xi_+^3}.
\end{equation}
Moreover, we remark that the Tracy-Widom limit for the general case is rescaled and it is not the same as the previous one we used for the null case. However they are just different by a simple scaling so that we do not emphasize this difference and still use the notation $ \TW $ to denote this distribution. Under this framework, our main result given in Corollary \ref{t.RateGeneralPopulation}.

Unlike the proof for the null case in the previous sections, we will not strictly follow the three-step strategy in the dynamical approach. Instead, we use the comparison theorem for the Green function flow, which is a method based on continuous interpolation, linearization and renormalization developed in \cite{lee2016tracy}.

\subsection{Local deformed Marchenko-Pastur law}
For completeness, we will briefly introduce the local deformed Marchenko-Pastur law. Though we will not give a detailed proof, we emphasize that the local law is an indispensable part to prove the Green function comparison Proposition \ref{p.GreenFunctionFlow}, which further leads to the edge universality.

For small nonnegative $ c,\epsilon \geq 0 $ and sufficiently large $ E_+<C<\infty $, we consider the domain
$$ \mathcal{D}(c,\epsilon) :=\left\{ z=E+\i\eta \in \C^+ : E_+-c \leq E \leq C,  N^{-1+\epsilon} \leq \eta \leq 1 \right\}. $$
We also denote $ \kappa=\kappa(E):=|E-E_+| $. Then we have the following estimates for the density and the Stieltjes transform of the deformed Marchenko-Pastur law.
\begin{lemma}[Theorem 3.1 in \cite{bao2015universality}]\label{l.EstiamtesDeformedMP}
Under the assumption \eqref{e.AssumptionPopulation}, there exists a constant $ c>0 $ such that
$$ \hat{\rho}_\fc(E) \sim \sqrt{E_+ - E}\ \ \ \ E \in [E_+-2c,E_+]. $$
Moreover, the Stieltjes transform $ \hat{m}_{\fc} $ satisfies the following: for $ z \in \mathcal{D}(c,0) $, we have
\begin{equation*}
|\hat{m}_{\fc}(z)| \sim 1,\ \ \ \ \im \hat{m}_{\fc}(z) \sim \left\{
\begin{aligned}
& \frac{\eta}{\sqrt{\kappa+\eta}} & \mbox{if} &\ E \geq E_++\eta,\\
& \sqrt{\kappa+\eta}, & \mbox{if} &\ E \in [E_+-c,E_++\eta).
\end{aligned}
\right.
.
\end{equation*}
\end{lemma}
The Green function and the Stieltjes transform are defined in the usual way:
$$ G_Q(z) := (Q-z)^{-1},\ \ \ m_Q(z) := \dfrac{1}{N}\tr G_Q(z). $$
Then we have the following local law for the separable sample covariance matrix $ Q $.
\begin{lemma}[Theorem 3.2 and Theorem 3.3 in \cite{bao2015universality}]\label{l.LocalDeformedMPLaw}
Under the assumption \eqref{e.AssumptionPopulation}, for any sufficiently small $ \epsilon>0 $, and for any (large) $ D>0 $, there exists $ N_0(D)>0 $ such that for any $ N \geq N_0(D) $ we have the following estimate uniformly in $ z \in \mathcal{D}(c,\epsilon) $:
$$ \P \left( |m_Q(z) - \hat{m}_{\fc}(z)| \leq \dfrac{N^\epsilon}{N\eta} \right) > 1-N^{-D}, $$
and
$$ \P \left( \max_{i,j}\left| (G_Q)_{ij}(z) - \delta_{ij}\hat{m}_{\fc}(z) \right| \leq N^\epsilon \left( \sqrt{\dfrac{\im \hat{m}_{\fc}(z)}{N\eta}} + \dfrac{1}{N\eta} \right) \right) > 1-N^{-D}. $$
\end{lemma}

It is clear to see that the estimates for the deformed Marchenko-Pastur law (Lemma \ref{l.EstiamtesDeformedMP}) and the local law (Lemma \ref{l.LocalDeformedMPLaw}) are greatly similar as the corresponding results for the null case (see e.g. \cite[Theorem 3.1]{pillai2014universality}). Heuristically, this implies the Tracy-Widom limit in the edge universality for the non-null case.

\subsection{Interpolation and Green function comparison}
In classical theory of random matrix universality, the tool needed to prove the edge universality is the Green function comparison theorem. The usual approach is to compare two ensembles with some moments matching conditions, and then use the construction of Gaussian divisible ensembles together with estimates of the local relaxation flow to remove the moments matching requirement. In this section, however, we do not follow this traditional step. Instead, we compare the Green function of a general ensemble with its corresponding null sample covariance matrix. This argument was first introduced in \cite{lee2015edge} to handle the deformed Wigner matrices, and used in \cite{lee2016tracy} to identify the Tracy-Widom limit for general separable sample covariance matrices. The basic idea is to introduce a time evolution that deforms the population matrices continuously to the identity and offset the change of the Green function by a renormalization of the matrix.

Recall the scaling constant $ \gamma_0 $ defined in \eqref{e.ScalingConstant}. We consider the following two rescaled matrices
$$ \tilde{\Sigma} := \gamma_0 \Sigma,\ \ \ \tilde{Q} := X^* \tilde{\Sigma}X. $$
We also denote the eigenvalues of $ \tilde{Q} $ by $ \tilde{\mu}_1 \leq \cdots \leq \tilde{\mu}_N $, and let $ L_+ := \gamma_0 E_+ $. We remark that in the literature about sample covariance matrices with general population (e.g. \cite{bao2015universality,el2007tracy,lee2016tracy}), the scaling of the Tracy-Widom distribution is chosen in the way such that it is the limit for the distribution of the (rescaled) largest eigenvalue of the matrix
$$ W:=\sqrt{\xi}(1+\sqrt{\xi})^{-4/3}X^*X. $$
Specifically, we order the eigenvalues of the matrix $ W $ by $ \lambda_1 \leq \cdots \leq \lambda_N $, and let $ M_+ $ denote the rightmost endpoint of the rescaled Marchenko-Pastur law for $ W $.

\begin{remark}
It has been shown in \cite[equation (1.9)]{bao2015universality} that
$$ \gamma_0 = \sqrt{\xi}(1+\sqrt{\xi})^{-4/3} + o(1). $$
This can be regarded as a good motivation for considering the scaling constant $ \gamma_0 $.
\end{remark}

For the diagonal population matrix $ \Sigma = \mbox{diag}(\sigma_j) $, we introduce the following time evolution $ t \mapsto (\sigma_j(t)) $ that deforms $ \Sigma $ to the identity matrix $ \Id $ and the Green function flow by
\begin{equation}\label{e.GreenFunctionFlow}
\dfrac{1}{\sigma_j(t)}=e^{-t}\dfrac{1}{\sigma_j(0)}+(1-e^{-t}),\ \ \ \tilde{Q}(t)=\gamma_0X^*\Sigma(t)X,\ \ \ m_{\tilde{Q}(t)}(z) := \dfrac{1}{N}\tr (\tilde{Q}(t)-z)^{-1}.
\end{equation}
Based on the local law Lemma \ref{l.EstiamtesDeformedMP} and Lemma \ref{l.LocalDeformedMPLaw}, and a delicate analysis for the time derivative of the Green function for $ \tilde{Q}(t) $, the Green function comparison theorem (see Proposition \ref{p.GreenFunctionFlow}) is proved in \cite{lee2016tracy}. We note that though the original estimate in \cite{lee2016tracy} is not explicit, a careful examination of the proof will reveal that the result is actually quantitative.

\begin{proposition}[Theorem 4.1 in \cite{lee2016tracy}]\label{p.GreenFunctionFlow}
Let $ \ep>0 $ and set $ \eta = N^{-2/3-\ep} $. Let $ E_1,E_2 \in \R $ satisfy $ E_1 < E_2 $ and $ |E_1|,|E_2| \leq N^{-2/3+\ep} $. Let $ F: \R \to \R $ be a smooth function satisfying
$$ \max_x |F^{(l)}(x)|(|x|+1)^{-C} \leq C,\ \ \ \ l=1,2,3,4. $$
Then for any (small) $ \de>0 $ and for sufficiently large $ N $ we have
\begin{multline*}
\left| \E F \left( N \int_{E_1}^{E_2} \im m_{\tilde{Q}}(x+L_++\i \eta) \d x \right) - \E F \left( N \int_{E_1}^{E_2} \im m_{W}(x+M_++\i \eta) \d x \right) \right|\\
\leq N^{-\frac{1}{3}+2\ep+\de}.
\end{multline*}
\end{proposition}

\subsection{Quantitative edge universality}
In this section we can finally prove Corollary \ref{t.RateGeneralPopulation}. The proof is based on our previous rate of convergence for the null case (Theorem \ref{t.Rate}) and the estimate on the comparison theorem for the Green function flow (Proposition \ref{p.GreenFunctionFlow}).
\begin{proof}[Proof of Corollary \ref{t.RateGeneralPopulation}]
We first remark that we are not supposed to use the rate of convergence for the null case (Theorem \ref{t.Rate}) and the triangle inequality in a naive way to derive the convergence rate for the general case. This is because Theorem \ref{t.Rate} is obtain by choosing the optimal parameters in the estimate \eqref{e.RateNullCase}, and the scale parameter $ \rho $ is also related to the scale in the Green function comparison (Proposition \ref{p.GreenFunctionFlow}). To illustrate this link more clearly, we first briefly review how Green function comparison is used to obtain the edge universality.

We introduce a smooth cutoff function $ K:\R \to \R $ satisfying
\begin{equation*}
K(x)=
\left\{
\begin{aligned}
& 1 &\ \mbox{if}\ & x \leq 1/9,\\
& 0 &\ \mbox{if}\ & x \geq 2/9.
\end{aligned}
\right.
\end{equation*}
and we also define the Poisson kernel $ \theta_\eta $, for $ \eta>0 $
$$ \theta_\eta(x) := \dfrac{\eta}{\pi(x^2+\eta^2)}. $$
Let $ E_*:=L_+ +\varphi^C N^{-2/3} $, and denote $ \chi_E := 1_{[E,E_*]} $. For $ \ep>0 $, let $ l:= \tfrac{1}{2}N^{-2/3-\ep} $ and $ \eta:=N^{-2/3-9\ep} $. Then for any (large) $ D>0 $, it is proved in \cite{lee2016tracy,pillai2014universality} that for large enough $ N $ we have
\begin{equation}
\E K\left( \tr (\chi_{E-l} * \theta_\eta(\tilde{Q}))  \right) \leq \P\left( \tilde{\mu}_N \leq E \right) \leq \E K\left( \tr (\chi_{E+l} * \theta_\eta(\tilde{Q}))  \right) + N^{-D}.
\end{equation}
Here the parameter $ l $ plays the same role as the $ \rho $ in \eqref{e.RateNullCase}, and therefore we have $ N^{-\ep}=N^{2/3}\rho $ and $ \eta=N^{-2/3}N^{6}\rho^{9} $. By the Green function comparison Proposition \ref{p.GreenFunctionFlow} and the \cite[Theorem 2.4]{lee2016tracy}, we have
\begin{multline*}
\P \left( N^{2/3}(\lambda_N - M_+) \leq s \right) - N^{-\frac{1}{3}+18\ep+\de} \leq \P \left( N^{2/3}(\tilde{\mu}_N - L_+) \leq s \right)\\
\leq \P \left( N^{2/3}(\lambda_N - M_+) \leq s \right) + N^{-\frac{1}{3}+18\ep+\de}.
\end{multline*}
This gives us
\begin{equation}\label{e.GeneralQuantEdgeUniversality}
\dk \left( \gamma_0 N^{2/3}(\mu_N - E_+),N^{2/3}(\lambda_N-M_+) \right) \leq N^{\de}N^{-1/3}N^{-12}\rho^{-18}.
\end{equation}
Combined with Theorem \eqref{e.RateNullCase}, by triangle inequality we finally obtain
\begin{align*}
&\quad \dk \left( \gamma_0 N^{2/3}(\mu_N - E_+),\TW \right)\\
& \leq \dk \left( \gamma_0 N^{2/3}(\mu_N - E_+),N^{2/3}(\lambda_N-M_+) \right) + \dk \left( N^{2/3}(\lambda_N-M_+),\TW \right)\\
& \leq N^{\de}\left( N^{-1/3}N^{-12}\rho^{-18} + N^{2/3}\rho + \dfrac{1}{N^{1/3}t} + \dfrac{1}{N^{18} \rho^{20}} + \dfrac{t}{N \rho} + \dfrac{1}{(N \rho)^2} + \dfrac{1}{N^{2/3}} \right).
\end{align*}
The optimal result is obtained now by choosing $ \rho=N^{-13/19} $ and $ t=N^{-6/19} $, which gives us
$$  \dk \left( \gamma_0 N^{2/3}(\mu_N - E_+),\TW \right) \leq N^{-\frac{1}{57}+\de}. $$
This completes the proof.
\end{proof}

Based on the Corollary \ref{t.RateGeneralPopulation} for diagonal population matrices $ \Sigma $ and general random matrices $ X $, we can easily obtain the following result for the case in which we can have a general population if the random matrix $ X $ is restricted to be Gaussian.

\begin{corollary}\label{c.RateGaussian}
Let $ Q := X^* \Sigma X $ be an $ N \times N $ separable sample covariance matrix, where $ X $ is an $ M \times N $ real random matrix with independent Gaussian entries satisfying \eqref{e.Assumption1}, and $ \Sigma $ is a real positive-definite deterministic $ M \times M $ matrix satisfying \eqref{e.AssumptionPopulation}. For any $ \ep>0 $ and large enough $ N $, we have
\begin{equation*}
\dk \left( \gamma_0 N^{2/3} (\mu_N - E_+),\TW \right) \leq N^{-\frac{1}{57} + \ep}
\end{equation*}
\end{corollary}
\begin{proof}
Under these assumptions, we know that the population matrix $ \Sigma $ is diagonizable, i.e. there exists an $ M \times M $ real diagonal matrix $ D $ and an $ N \times N $ orthogonal matrix $ U $ such that $ \Sigma = U^*DU $. Since $ X $ is a matrix whose entries are independent Gaussian random variable, we know that $ UX $ is also a real random matrix with Gaussian entries satisfying the assumption \eqref{e.Assumption1}. Therefore, by applying our Corollary \ref{t.RateGeneralPopulation} to the matrix $ X^*\Sigma X = (UX)^*D(UX) $, we will get the desired result.
\end{proof}

\small
\bibliographystyle{abbrv}
\bibliography{CovarianceTW}

\end{document}